\documentclass[10pt,oneside]{article}
\usepackage{amsmath, amsfonts, amsthm, amssymb, amscd}
\usepackage[mathcal]{euscript} 
\usepackage{dsfont, pxfonts, verbatim}
\usepackage{mathrsfs}  
\newtheorem{theorem}{Theorem}[section]
 
\newtheorem{lemma}[theorem]{Lemma}
 
\newtheorem{definition}[theorem]{Definition} 
\newtheorem{remark}[theorem]{Remark}  
\numberwithin{equation}{section}

\def\an    {{A^\natural}}  
\def\anb   {\overline{\an}}
\def\ane   {{A^\natural_*}} 
\def\ab    {\overline a}

\def\be    {\beta}  

\def\ban   {\underline{\an}}

\def\bmu   {\underline{\mu}} 
\def\br    {\underline{r}}

\def\del   {{\partial}}

\def\eps   {\varepsilon}

\def\lam        {\lambda} 
\def\Lamalpha   {\Lambda_\alpha}

\def\lna   {\lambda^\natural} 
\def\lz    {\lambda_0}

\def\mub   {\overline{\mu}}

\def \RR   {\mathbb{R}}

\def\rb    {\overline r}

\def\sgn   {\text{\rm sgn}}  
\def\SS    {{\mathcal S}}

\def\talpha {\tilde\alpha}

\def\vf    {{\varphi^\flat}}  
\def\vfalpha  {{\varphi^\flat_\alpha}}

\def\vfz   {{\varphi^\flat_0}}

\def\vmn   {\varphi^{-\natural}}
\def\vn    {{\varphi^\natural}}

\def\vsalpha {{\varphi^\sharp_\alpha}}  
\def\vsz   {{\varphi^\sharp_0}}

\def \be   {\begin{equation}}
\def \ee   {\end{equation}}

\begin{document}

\title{Existence and qualitative properties of kinetic functions generated by 
 diffusive-dispersive regularizations}

\author{ 
Philippe G. LeFloch\footnote{Laboratoire Jacques-Louis Lions \& Centre National de la Recherche Scientifique, 
Universit\'e Pierre et Marie Curie (Paris 6), 4 Place Jussieu, 75252 Paris, France. E-mail : {\sl pgLeFloch@gmail.com.} 
Blog: {\sl http://PhilippeLeFloch.wordpress.com} 
\newline 
2000\textit{\ AMS Subject Classification.} 35l65, 76N10, 35B45. 
}
}
\date{May 14, 2010}
\maketitle

\begin{abstract} We investigate the properties of 
traveling wave solutions to hyperbolic conservation laws augmented with diffusion and dispersion, and 
review the existence and qualitative properties of the associated kinetic functions, 
which characterize the class of admissible shock waves selected by such regularizations. 
\end{abstract} 

\tableofcontents
  

\section{Introduction}
\label{IN}

In this paper, we investigate the properties of
 traveling wave solutions to hyperbolic conservation laws augmented with diffusion and dispersion, and 
reviews the existence and qualitative properties of the associated kinetic functions. Such a function 
 characterizes the class of admissible shock waves, both compressive and undercompressive,
  selected by a given regularization. 
Building on the pioneering papers \cite{Slemrod1,Truskinovsky2,AbeyaratneKnowles2,LeFloch-ARMA}, 
the mathematical research on undercompressive shocks generated by diffusive-dispersive limits 
developed intensively in the last fifteen years. 
For background on this topics and further material, 
we refer the reader to the reviews \cite{LeFloch-Freiburg,LeFloch-book,LeFloch-Oslo} and the extensive literature cited therein. 
The present review restrict attention to  traveling waves and to a class of scalar equations. 

The kinetic relation can be defined as follows. Recall that classical compressive shocks with a given left-hand state $u_-$
(and wave family, when systems of equations are considered) 
form a one-parameter family of solutions, parametrized by their right-hand state $u_+$. By contrast, 
given any left-hand state $u_-$ (and wave family), there typically exists a single undercompressive shock, and 
the kinetic function $\varphi^\flat$ precisely determines the right-hand state 
$$
u_+ = \varphi^\flat(u_-)
$$
 as a function of the left-hand side.

The fundamental questions of interest are the following ones: 
do there exist traveling wave solutions associated with classical and/or with 
nonclassical shock waves~? 
Can one associate a kinetic function to the given model~?
If so, is this kinetic function monotone ? What is the behavior of arbitrarily small shocks~? 
How does the kinetic function depend upon the parameters?
 
Answers to these questions were obtained first for the cubic flux function, 
by deriving {\sl explicit} formulas for the kinetic function in 
 Shearer et al. \cite{JMS} and 
Hayes and LeFloch \cite{HayesLeFloch-scalar}. 
General flux-functions and general regularization were covered
 by Bedjaoui and LeFloch in the series of papers \cite{BedjaouiLeFloch1}--\cite{BedjaouiLeFloch5}. 
 
 More generally, the existence and properties of traveling waves for 
the nonlinear elasticity and the Euler equations 
are known in both the hyperbolic \cite{SchulzeShearer,BedjaouiLeFloch2} 
and the hyperbolic-elliptic regimes \cite{Truskinovsky2,ShearerYang,Benzoni,BedjaouiLeFloch3}.
For all other models, only partial results on traveling waves are available. 
 
The existence of nonclassical traveling wave solutions for the thin liquid film model is proven 
by Bertozzi and Shearer in \cite{BertozziShearer}. 
For this model, no qualitative information on the properties of these traveling waves is known, 
and, in particular, the existence of the kinetic relation has not been rigorously established yet.
The kinetic function was recently determined numerically in LeFloch and Mohamadian \cite{LeFlochMohamadian}. 
For the $3\times 3$ Euler equations, we refer to \cite{BedjaouiLeFloch5}.

Finally, we also recall that the Van de Waals model admits {\sl two} inflection points and leads 
to {\sl multiple} traveling wave solutions. Although the physical significance
of the ``second'' inflection point is questionable, 
given that this model 
is extensively used in the applications it is important to investigate whether additional features
arise. Indeed, it is established in
\cite{BCCL} that {\sl non-monotone nonclassical} traveling wave profiles exist, and that 
a single kinetic function is not sufficient to single out the physically relevant solutions. 

An outline of this paper is as follows. In Section~2, we briefly discuss the case of the diffusion model, 
while the rest of the paper is concerned with the diffusion-dispersion model. 
We then begin with the case of a cubic flux-function for which explicit formulas can be derived. 
The main results are stated in Section~4 for general flux-functions having one inflection point. 
Sections~5 and 6 are concerned with the derivations of key properties of the traveling waves and kinetic function, 
corresponding to a fixed shock speed and to a fixed diffusion over dispersion ratio, respectively.  

\section{Traveling waves associated with the nonlinear diffusion model}

Consider the scalar conservation law 
\be
\del_t u + \del_x f(u) = 0, \quad  u=u(x,t) \in \RR, 
\label{1.1} 
\ee
where $f:\RR \to \RR$ is a smooth mapping. We begin, in this section, with the {\sl 
nonlinear diffusion model} 
\be
\del_t u + \del_x f(u) = \eps \, \bigl(b(u) \, u_x \bigr)_x, 
\quad u = u^\eps(x,t) \in \RR, 
\label{1.2} 
\ee
where $\eps>0$ is a small parameter. The {\sl diffusion function\/}  
$b: \RR \to \RR_+$ is assumed to be smooth and bounded below: 
\be
b(u) \geq \bar b >0, 
\label{1.3} 
\ee
so that the equation \eqref{1.2} is {\sl uniformly parabolic}. 
We are going to establish that the shock set associated with 
the traveling wave solutions of \eqref{1.2} coincides with the one 
described by Oleinik entropy inequalities (see \eqref{1.9}, below).

Recall that a {\sl traveling wave} of \eqref{1.2} is a solution depending 
only upon the variable 
\be
y := {x - \lam \, t \over \eps}  
\label{1.4} 
\ee
for some constant speed $\lam$. 
Note that, after rescaling, 
the corresponding {\sl trajectory} $y \mapsto u(y)$ is independent of the parameter 
$\eps$.  
Fixing the left-hand state $u_-$ we search for traveling waves of 
\eqref{1.2} {\sl connecting $u_-$ to some state $u_+$}, 
that is, solutions $y \mapsto u(y)$ of the ordinary differential equation 
\be 
- \lam \, u_y + f(u)_y = \bigl( b(u) \, u_y \bigr)_y   
\label{1.5} 
\ee
satisfying the boundary conditions
\be 
\lim_{y \to -\infty} u(y) = u_-, 
\quad 
\lim_{y \to +\infty} u(y) = u_+, 
\quad 
\lim_{|y| \to +\infty} u_y(y) =  0. 
\label{1.6}   
\ee 

In view of \eqref{1.6} the equation \eqref{1.5} can be integrated once: 
\be
b(u(y)) \, u_y(y) = -\lam \, (u(y) - u_-) + f(u(y)) - f(u_-), \quad 
y \in \RR. 
\label{1.7}  
\ee
The Rankine-Hugoniot condition 
\be
- \lam \, (u_+ - u_-) + f(u_+) - f(u_-) = 0 
\label{1.8}  
\ee
follows by letting $y \to +\infty$ in \eqref{1.7}.  
The equation \eqref{1.7} is an ordinary differential equation 
(O.D.E) on the real line. 
The qualitative behavior of the solutions is easily determined, 
as follows.

\begin{theorem}[Diffusive traveling waves] 
\label{Theorem1.1} 
Consider the scalar conservation law \eqref{1.1} with general flux-function $f$
together with the diffusive model \eqref{1.2}. 
Fix a left-hand state $u_-$ and a right-hand state $u_+ \neq u_-$.  
Then, there exists a traveling wave of $\eqref{1.7}$ associated with 
the nonlinear diffusion model \eqref{1.2} if and only if $u_-$ and $u_+$
satisfy {\sl Oleinik entropy inequalities in the strict sense}, that 
is: 
\be
{f(v) - f(u_-) \over v - u_-} >  
{f(u_+) - f(u_-) \over u_+ - u_-} 
\quad 
\text{ for all } v \text{ lying strictly between $u_-$ and $u_+$.} 
\label{1.9} 
\ee
\end{theorem}

\begin{proof} 
All the trajectories of 
interest are bounded, i.e., cannot escape to infinity. 
Namely, the shock profile satisfies the equation
\be
u' = {u-u_- \over b(u)} \, \left(
     {f(u) - f(u_-) \over u - u_-} 
   - {f(u_+) - f(u_-) \over u_+ - u_-}\right). 
\label{1.10}
\ee
It is not difficult to see that the solution exists and connects 
monotonically $u_-$ to $u_+$ provided Oleinik entropy inequalities hold
and the right-hand side of \eqref{1.10} keeps (strictly) a constant sign 
(except at the end point $y= \pm \infty$ where it vanishes).  
\end{proof}

We define the {\sl shock set\/} associated with the nonlinear diffusion model 
as 
$$
S(u_-) := \bigl\{ u_+ \, / \, \text{ there exists a solution of } \eqref{1.6}--\eqref{1.8} \bigr\}. 
$$ 
From Theorem~\ref{Theorem1.1} one can deduce the following.

\begin{theorem}[Shock set based on diffusive limits] 
\label{Theorem1.2}  
Consider the scalar conservation law \eqref{1.1} when 
the flux $f$ is convex, concave-convex, or convex-concave (see \eqref{3.2}, below).   
Then, for any $u_-$, the shock set $S(u_-)$ associated with the 
nonlinear diffusion model \eqref{1.2} and \eqref{1.3} is independent of 
the diffusion function $b$, and the closure of $S(u_-)$ 
coincides with the shock set characterized 
by Oleinik entropy inequalities (or, equivalently, Lax shock inequalities). 
\end{theorem}

\begin{remark}
\label{Remark1.3} 
The conclusions of Theorem~\ref{Theorem1.2} do not hold for 
more general flux-functions. This is due to the fact that a strict inequality 
is required in \eqref{1.9} for the existence of the traveling waves.
The set based on traveling waves may be strictly smaller than 
the one based on Oleinik entropy inequalities. 
\end{remark}


\section{Kinetic functions associated with cubic flux-func\-tions}

Investigating traveling wave solutions of diffusive-dispersive regularizations 
of \eqref{1.1} is considerably more involved than what was done in Section~2. 
Besides proving the existence of associated (classical and nonclassical) traveling waves 
our main objective will be to derive the corresponding kinetic functions 
for nonclassical shocks.

To explain the main difficulty and ideas it will
be useful to treat first, in the present section, the specific 
{\sl diffusive-dispersive model with cubic flux}  
\be 
\del_t u + \del_x u^3  
= \eps \, u_{xx} + \delta \, u_{xxx},   
\label{2.1}
\ee
which, formally as $\eps, \delta \to 0$, converges to the 
{\sl conservation law with cubic flux} 
\be 
\del_t u + \del_x u^3 = 0. 
\label{2.2}
\ee 
We are interested in the singular limit $\eps \to 0$ in \eqref{2.1} when the ratio 
\be
\alpha = {\eps \over \sqrt\delta} 
\label{2.3}
\ee
is kept constant. We assume also that the dispersion coefficient $\delta$ 
is positive. Later, in Theorem~\ref{Theorem3.5} below, 
we will see that all traveling waves are classical when ${\delta < 0}$
which motivates us to restrict attention to ${\delta>0}$.

We search for traveling wave solutions of \eqref{2.1} depending on the 
rescaled variable 
\be
y : = \alpha \, {x-\lam \, t \over \eps} = {x-\lam \, t \over \sqrt \delta }. 
\label{2.4}
\ee 
Proceeding along the same lines as those in Section~2 we find that 
a traveling wave $y \mapsto u(y)$ should satisfy   
\be 
- \lam \, u_y + (u^3)_y = \alpha \, u_{yy} + u_{yyy},  
\label{2.5}
\ee
together with the boundary conditions
\be
\aligned 
& \lim_{y\to \pm\infty} u(y) = u_\pm, 
\\
& \lim_{y\to\pm\infty} u_y(y) = \lim_{y\to\pm\infty} u_{yy}(y) = 0, 
\endaligned 
\label{2.6}
\ee 
where $u_- \neq u_+$ and $\lam$ are constants. 
Integrating \eqref{2.5} once we obtain 
\be
\alpha \, u_y(y) + u_{yy}(y) = - \lam \, (u(y) - u_-) + u(y)^3 - u_-^3,
\quad y \in \RR, 
\label{2.7}
\ee
which also implies  
\be
\lam = {u_+^3 - u_-^3 \over u_+ - u_-} =  u_-^2 + u_- \, u_+ + u_+^2.   
\label{2.8}
\ee 

To describe the family of traveling waves 
it is convenient to fix the left-hand state
(with for definiteness $u_- > 0$) 
and to use the speed $\lam$ as a parameter. 
Given $u_-$, there is a range of speeds, 
$$ 
\lam \in (3 \, u_-^2/4, 3 \, u_-^2),  
$$ 
for which the line passing through the point with coordinates $(u_-, u_-^3)$
and with slope $\lam$ intersects the graph of the flux 
$f(u) := u^3$ at three distinct points. 
For the discussion in this section we restrict attention to this 
situation, which is most interesting.  
There exist {\sl three equilibria} at which the 
right-hand side of \eqref{2.7} vanishes. The notation   
$$ 
u_2 < u_1 < u_0 := u_- 
$$ 
will be used, where $u_2$ and $u_1$ are the two distinct roots of the polynomial 
\be 
u^2 + u_0 \, u + u_0^2 = \lam.  
\label{2.9}
\ee
Observe in passing that $u_2 + u_1 + u_0 = 0$.

Consider a trajectory $y \mapsto u(y)$ leaving from $u_-$ at 
$-\infty$. We want to determine which point, among $u_1$ or $u_2$,   
the trajectory will reach at $+\infty$. Clearly, the trajectory 
is associated with a so-called {\sl classical shock} if it reaches $u_1$ 
and with a so-called {\sl nonclassical shock} if it reaches $u_2$.  
Accordingly, we will refer to it as a {\sl classical trajectory}
or as a {\sl nonclassical trajectory}, respectively.


We reformulate \eqref{2.7} as a differential system of two equations,  
\be
{d \over dy} \begin{pmatrix} u \\ v \end{pmatrix} 
= K(u,v), 
\label{2.10}
\ee
where 
\be 
K(u,v) = \begin{pmatrix} v \\ -\alpha \, v + g(u, \lam) -g(u_-, \lam)
\end{pmatrix}, 
\quad 
g(u, \lam) = u^3 - \lam \, u. 
\label{2.11}
\ee
The function $K$ vanishes precisely at the three equilibria  
$(u_0, 0)$, $(u_1, 0)$, and $(u_2, 0)$ of \eqref{2.10}. 
The eigenvalues of the Jacobian matrix of $K(u,v)$ at any point $(u,0)$ are 
$-\alpha/2~\pm~\sqrt{\alpha^2/4~+~g_u'(u, \lam)}$. 
So we set 
\be
\aligned 
\bmu(u) 
& = {1 \over 2} \left(-\alpha - \sqrt{\alpha^2 + 4 \, (3 \, u^2 - \lam)}\right),
\\
\mub(u) 
& = {1 \over 2} \left(-\alpha + \sqrt{\alpha^2 + 4 \, (3 \, u^2 - 
\lam)}\right). 
\endaligned 
\label{2.12}
\ee
At this juncture, we recall the following standard definition and 
result. (See the bibliographical notes for references.)

\begin{definition}[Nature of equilibrium points]
\label{Definition2.1}  
Consider a differential system of the form $\eqref{2.10}$ where $K$ is 
a smooth mapping. Let $(u_*,v_*) \in \RR^2$ be an equilibrium point, 
that is, a root of $K(u_*,v_*)=0$. 
Denote by $\bmu=\bmu(u_*,v_*)$ and $\mub=\mub(u_*,v_*)$ the two 
(real or complex) eigenvalues of the Jacobian matrix of $K$ at 
$(u_*,v_*)$, and suppose that a basis of corresponding 
eigenvectors $\br(u_*,v_*)$ and $\rb(u_*,v_*)$ exists. 
Then, the equilibrium $(u_*,v_*)$ is called 
\begin{enumerate} 
\item a {\sl stable point} if $Re(\bmu)$ and $Re(\mub)$ are both negative,  
\item a {\sl saddle point} if $Re(\bmu)$ and $Re(\mub)$
have opposite sign, 
\item or an {\sl unstable point} if $Re(\bmu)$ and $Re(\mub)$
are both positive. 
\end{enumerate} 
Moreover, a stable or unstable point is called a {\sl node} if   
the eigenvalues are real and a {\sl spiral} if they are complex conjugate. 
\end{definition}

\begin{theorem}[Local behavior of trajectories]
\label{Theorem2.2} 
Consider the differential system $\eqref{2.10}$ under the same assumptions 
as in Definition 2.1. 
If $(u_*,v_*)$ is a saddle point, there are two 
trajectories defined on some interval $(-\infty, y_*)$ and
two trajectories defined on some interval  $(y_*, +\infty)$ 
and converging to $(u_*,v_*)$ at $-\infty$ and $+\infty$, respectively. 
The trajectories are tangent to the eigenvectors $\br(u_*, v_*)$ and $\rb(u_*, v_*)$, 
respectively. 
\end{theorem}

Returning to \eqref{2.11} and \eqref{2.12} we conclude that, 
since $g_u'(u, \lam) = 3u^2-\lam$ is positive at both $u=u_2$ and $u=u_0$, 
we have 
$$
\bmu(u_0) < 0 < \mub(u_0), \quad \bmu(u_2) < 0 < \mub(u_2). 
$$ 
Thus both points $u_2$ and $u_0$ are {\sl saddle points.\/} 
On the other hand, since we have ${g_u'(u_1, \lam) < 0}$, the point $u_1$ is 
stable: it is a {\sl node\/} if $\alpha^2 + 4 \, (3 \, u_1^2-\lam) \geq 0$ 
or a {\sl spiral\/} if ${\alpha^2 + 4 \, (3 \, u_1^2 - \lam) < 0}$. 
In summary,  
for the system \eqref{2.10}-\eqref{2.11} 
\be
\aligned 
& \text{$u_2$ and $u_0$ are saddle points and}
\\ 
& \text{$u_1$ is a stable point (either a node or a spiral).} 
\endaligned 
\label{2.13}
\ee


In the present section we check solely that, in some range of 
the parameters $u_0$, $\lam$, and $\alpha$, there exists 
a {\sl nonclassical trajectory\/} connecting the two saddle points $u_0$ 
and $u_2$. {\sl Saddle-saddle connections} are not ``generic'' and, 
as we will show, arise only 
when a special relation (the kinetic relation) holds between 
$u_0$, $\lam$, and $\alpha$ or, equivalently, between 
$u_0$, $u_2$, and $\alpha$; see \eqref{2.15} below. 

For the cubic model \eqref{2.1} an {\sl explicit formula\/} is 
now derived for the nonclassical trajectory. 
Motivated by the fact that the function $g$ in \eqref{2.11} is a cubic, 
we a~priori assume that $v=u_y$ is a {\sl parabola in the variable $u$}. 
Since $v$ must vanish at the two equilibria we write  
\be
v(y) = a \, (u(y) - u_2) \, ( u(y) - u_0), \quad y \in \RR, 
\label{2.14}
\ee
where $a$ is a constant to be determined. 
Substituting \eqref{2.14} into \eqref{2.10}-\eqref{2.11}, we obtain an expression of 
$v_y$: 
$$
\aligned 
v_y & = - \alpha \, v + u^3 - u_0^3 - \lam \, ( u - u_0) \\ 
    & = - \alpha \, v + (u - u_2) \, ( u - u_0) \, ( u + u_0 + u_2) \\ 
    & = v \, \bigl(- \alpha + {1 \over a} \, ( u + u_0 + u_2) \bigr).   
\endaligned 
$$
But, differentiating \eqref{2.14} directly we have also  
$$
\aligned 
v_y & = a \, u_y \, ( 2 \, u - u_0 - u_2) 
\\
    & = a \, v \, ( 2 \, u - u_0 - u_2).  
\endaligned 
$$
The two expressions of $v_y$ above coincide if we choose  
$$
{1 \over a}  = 2 \, a, 
\quad 
- \alpha + {1 \over a} \, (u_0 + u_2) = - a ( u_0 + u_2).  
$$
So, $a = 1/\sqrt{2}$ (since clearly we need $v<0$) 
and the three parameters $u_0$, $u_2$, and $\alpha$ satisfy
the {\sl explicit relation} 
\be
u_2 = - u_0 + {\sqrt{2} \over 3} \, \alpha.
\label{2.15}
\ee 
Since $u_1 = - u_0 - u_2$ we see that the trajectory \eqref{2.14} is 
the saddle-saddle connection we are looking for, only if 
$u_2 < u_1$ as expected, that is, only if 
\be
u_0 > {2 \sqrt{2} \over 3} \, \alpha.
\label{2.16}
\ee

Now, by integrating \eqref{2.14}, it is not difficult to arrive at the following 
{\sl explicit formula for the nonclassical trajectory}:   
\be
\aligned 
u(y) & = {u_0 + u_2 \over 2} - {u_0 - u_2  \over 2} \, 
         \tanh \Big( {u_0 - u_2  \over 2 \sqrt{2} } \, y\Big)
\\
     & = { \alpha \over 3 \sqrt{2} } 
         - \Big(u_- - {\alpha \over 3 \sqrt{2}} \Big) \, 
         \tanh \Big( \Big(u_- - {\alpha \over 3 \sqrt{2}} \Big)  \, {y 
         \over \sqrt{2}} \Big).
\endaligned 
\label{2.17}
\ee
We conclude that, given any left-hand state $u_0  > 2 \sqrt{2}\, 
\alpha/3$, there exists a saddle-saddle connection 
connecting $u_0$ to $- u_0 + \sqrt{2} \, \alpha/3$
which is given by \eqref{2.17}. Later, in Section~4 and followings, 
we will prove that the trajectory just found is actually the 
{\sl only\/} saddle-saddle trajectory leaving from 
${u_0 > 2\sqrt{2} \, \alpha/3}$ and that no such trajectory exists
when $u_0$ is below that threshold.

Now, denote by $\SS_\alpha(u_-)$ the set of all 
right-hand states $u_+$ attainable through a diffusive-dispersive traveling wave
of $\eqref{2.1}$ with $\delta>0$ and $\eps/ \sqrt{\delta} = \alpha$ fixed. 
In the case of the equation \eqref{2.1} the results to be established 
in the following sections can be summarized as follows.  

\begin{theorem}[Kinetic function and shock set for the cubic flux]
\label{Theorem2.3} 
The {\sl kinetic function} associated with the diffusive-dispersive 
model $\eqref{2.1}$ is 
\be
\vfalpha(u_-) 
= \begin{cases}  
 - u_- - \talpha/2,  
   &         u_- \leq - \talpha, 
   \\
   - u_- /2,
   &         |u_-| \leq \talpha, 
   \\ 
 - u_- + \talpha/2,  
   &         u_- \geq \talpha, 
\end{cases}
\label{2.18}
\ee
with $\talpha := 2 \, \alpha \, \sqrt{2} /3$, 
while the corresponding {\sl shock set} is   
\be
\SS_\alpha(u_-) \, = \, \begin{cases} 
(u_-, \talpha/2] \cup \big\{-u_- - \talpha/2 \big\},  
              &   u_-  \leq -\talpha, 
\\ 
[- u_- /2, u_-),     &  - \talpha \leq u_- \leq \talpha, 
\\
\big\{-u_- + \talpha/2 \big\} \cup [- \talpha/2, u_-),  
              &   u_-  \geq \talpha. 
\end{cases} 
\label{2.19}
\ee
\end{theorem}

In agreement with the general theory  
of the kinetic function,  
\eqref{2.18} is monotone decreasing and lies between the limiting functions 
$\vn(u):= -u/2$ and $\vfz(u) := -u$. Depending on $u_-$ the shock set 
can be either an interval or the union of a point and an interval. 
 

Consider next the {\sl entropy dissipation} associated with the nonclassical shock:   
\be
\aligned 
E(u_-; \alpha, U) 
:= & - \bigl(\vfalpha(u_-)^2 + \vfalpha(u_-) \, 
       u_- + u_-^2\bigr) \, \bigl(U(\vfalpha(u_-)) - U(u_-) \bigr) 
       \\
   & + F(\vfalpha(u_-)) - F(u_-), 
\endaligned 
\label{2.20}
\ee 
where $(U, F)$ is any convex entropy pair of the equation \eqref{2.2}. 
By multiplying \eqref{2.5} by $U'(u(y))$ and 
integrating over $y \in \RR$ we find the equivalent expression  
\be
\aligned 
E(u_-; \alpha, U) 
& = \int_\RR U'(u(y)) \, \big( \alpha \, u_{yy}(y) + u_{yyy}(y) \big) \, dy 
\\ 
& = \int_\RR \bigl( - \alpha \, U''(u) \, u_y^2 
+ U'''(u) \, u_y^3 / 2 \bigr) \, dy. 
\endaligned
\label{2.21}
\ee
So, the sign of the entropy dissipation can also be determined from 
the explicit form \eqref{2.17} of the traveling wave.

\begin{theorem}[Entropy inequalities]
\label{Theorem2.4}  
\begin{enumerate}
\item For the quadratic entropy 
$$
U(z) = z^2/2, \quad z \in \RR,
$$ 
the entropy dissipation
${E(u_-; \alpha, U)}$ is non-positive for all real $u_-$ and all 
${\alpha \geq 0}$. 
\item For all convex entropy $U$ 
the entropy dissipation $E(u_-; \alpha, U)$ is non-positive for all  
$\alpha>0$ and all $|u_-| \leq 2 \sqrt{2} \, \alpha/3$. 
\item Consider $|u_-| >2 \sqrt{2} \, \alpha/3$ and any (convex) entropy 
$U$ whose third derivative is sufficiently small, specifically  
\be
\bigl(|u_-| - \alpha/ (3 \sqrt{2}) \bigr)^2 \, | U'''(z) | 
\leq 2 \alpha \, \sqrt{2} \, U''(z),  \quad z \in \RR.  
\label{2.22}
\ee
Then, the entropy dissipation $E(u_-; \alpha, U)$ is also non-positive.  
\item Finally given any $|u_-| > 2 \sqrt 2 \, \alpha/3$ 
           there exists infinitely many strictly convex 
           entropies for which $E(u_-;\alpha, U)$ is positive. 
\end{enumerate} 
\end{theorem}

\begin{proof}  
When $U$ is quadratic (with $U'' \geq 0$ and $U''' \equiv 0$) we already 
observed that Item 1 follows immediately from \eqref{2.21}. 
The statement Item 2 is also obvious since the function $\vf$ reduces to a 
classical value in the range under consideration. 
Under the condition \eqref{2.22} the {\sl integrand\/} of \eqref{2.21} is 
non-positive, as follows from the inequality (see \eqref{2.14}) 
$$
|u_y| \leq {1 \over 4 \sqrt{2}} \, (u_0 - u_2)^2 = {1 \over \sqrt{2}} 
\, \bigl(u_- - \alpha / (3 \sqrt{2})\bigr)^2. 
$$ 
This implies the statement Item 3. 
Finally, to derive Item 4 we use the (Lipschitz continuous) {\sl Kruzkov 
entropy pairs} 
\be
U_k(z) := |z-k|, \quad F_k(z) := \sgn (z-k) (z^3 - k^3), \quad z \in \RR, 
\label{2.23}
\ee
with the choice $k=-u_-/2$. We obtain 
$$ 
E(u_-; \alpha, U_k) 
= {3 \over 4} \, |u_-| \, \bigl( |u_-| - 2 \alpha \, \sqrt{2} 
/3\bigr)^2 >0. 
$$
By continuity, $E(u_-; \alpha, U_k)$ is also strictly positive for all $k$ 
in a small neighborhood of $-u_-/2$. The desired conclusion follows by observing 
that any smooth convex function can be represented by a weighted sum of Kruzkov entropies. 
\end{proof}


\begin{remark}
\label{Remark2.5} 
We collect here the explicit expressions of some functions 
associated with the model \eqref{2.1}.  
From now on we restrict attention to the entropy pair
$$
U(u) = u^2/2, \quad F(u) = 3 \, u^4/4. 
$$
First of all, recall that for the equation \eqref{2.2} 
the following two functions   
\be
\vn(u) = - {u \over 2}, \quad \vfz(u) = -u, \quad u \in \RR. 
\label{2.24}
\ee
determine the admissible range of the kinetic functions.

We define the {\sl critical diffusion-dispersion ratio}  
\be 
A(u_0, u_2) = {3 \over \sqrt{2}} \, (u_0 + u_2)
\label{2.25}
\ee
for $u_0 \geq 0$ and $u_2 \in (-u_0, -u_0/2)$
and 
for $u_0 \leq 0$ and $u_2 \in (-u_0/2, -u_0)$.  
In view of Theorem~\ref{Theorem2.3} (see also \eqref{2.15}), 
a nonclassical trajectory connecting $u_0$ to $u_2$
exists if and only if the parameter $\alpha= \eps/\sqrt{\delta}$  
equals $A(u_0, u_2)$.
The function $A$ increases monotonically in $u_2$ from the value $0$
to the {\sl threshold diffusion-dispersion ratio}  ($u_0 >0$) 
\be
\an(u_0) = {3 \, u_0  \over 2 \, \sqrt{2}}.
\label{2.26}
\ee 
For each fixed state $u_0>0$ there exists a nonclassical trajectory 
leaving from $u_0$ if and only if $\alpha$ is less than $\an(u_0)$.
On the other hand, for each fixed $\alpha$ there exists a nonclassical 
trajectory leaving from $u_0$ if and only if the left-hand state $u_0$ 
is greater than $\an^{-1}(\alpha)$. The function $\an$ is a linear function 
(for $u_0>0$) with range extending therefore from $\ban=0$ to $\anb=+\infty$.  
\end{remark}

\begin{remark}\label{Remark2.6} 
It is straightforward to check that if \eqref{2.1} is replaced with the more 
general equation
\be 
\del_t u + \del_x \bigl( K \, u^3 \bigr)   
= \eps \, u_{xx} + \delta \, C \, u_{xxx},   
\label{2.27}
\ee
where $C $ and $K$ are positive constants, then \eqref{2.26} becomes
\be
\an(u_0) = {3 \, u_0  \over 2 \, \sqrt{2}} \sqrt{K \, C}.
\label{2.28}
\ee 
\end{remark}

\begin{remark}\label{Remark2.7} 
Clearly, there is a one-parameter family of traveling waves 
connecting the same end states: If $u=u(y)$ is a solution of \eqref{2.5} and \eqref{2.6}, then the  
translated function $u=u(y + b)$ ($b \in \RR$) satisfies the same 
conditions. However, one could show that the {\sl trajectory\/} in the phase plane 
connecting two given end states is {\sl unique.\/} 
\end{remark} 


\section{Kinetic functions associated with general flux-func\-tions} 

Consider now the general {\sl diffusive-dispersive conservation law} 
\be 
\del_t u + \del_x f(u) 
= 
\eps \, \bigl(b(u) \, u_x \bigr)_x 
+ \delta \, \bigl(c_1(u) \, (c_2(u) \, u_x)_x\bigr)_x, 
\quad u = u^{\eps, \delta}(x,t), 
\label{3.1}
\ee
where the diffusion coefficient $b(u) > 0$ and dispersion coefficients $c_1(u), c_2(u) > 0$ 
are given smooth functions.  
We assume that $f:\RR \to \RR$ is a {\sl concave-convex\/} function satisfying, by definition, 
\be
\aligned 
& u \, f''(u) > 0  \quad \text{ for all } u \neq 0, \\
& f'''(0) \neq 0, \quad 
     \lim_{|u| \to +\infty} f'(u) = + \infty.  
\endaligned 
\label{3.2}
\ee  
We are interested in the singular limit $\eps \to 0$ when 
$\delta>0$ and the ratio 
$\alpha = \eps /\sqrt\delta$ is kept constant. The 
limiting equation associated with \eqref{3.1}, formally,
is the scalar conservation law 
$$ 
\del_t u + \del_x f(u) = 0, \quad u=u(x,t) \in \RR.
$$  
It can be checked that the entropy inequality 
$$ 
\del_t U(u) + \del_x F(u) \leq 0  
$$ 
holds, provided the entropy pair $(U, F)$ is chosen such that 
\be
U''(u) := {c_2(u) \over c_1(u)}, \quad
F'(u)  := U'(u) \, f'(u), \quad u \in \RR,
\label{3.3}
\ee   
which we assume in the rest of this paper. 
Since $c_1, c_2>0$ the function $U$ is strictly convex. 


Given two states $u_\pm$ and the corresponding propagation speed 
$$ 
\lam = \ab(u_-, u_+): = \begin{cases} 
{f(u_+) - f(u_-) \over u_+ - u_-},  &          u_+ \neq u_-, 
\\
f'(u_-),                            &          u_+ = u_-, 
\end{cases} 
$$ 
we search for traveling wave solutions $u=u(y)$ of \eqref{3.1} 
depending on the rescaled variable 
${y: = (x-\lam \, t)\, \alpha/\eps}$.  
Following the same lines as those in Sections~1 and 2 we find that 
the trajectory satisfies 
\be 
c_1(u) \, ( c_2(u) \, u_y)_y + \alpha \, b(u)u_y 
= -\lam \, (u - u_-) + f(u) - f(u_-), 
\quad u = u(y), 
\label{3.4}
\ee 
and the boundary conditions
$$ 
\lim_{y \to \pm \infty} u(y) = u_\pm, 
\quad 
\lim_{y \to \pm \infty} u_y(y) = 0. 
$$

Setting now 
$$
v = c_2(u) \, u_y,
$$
we rewrite \eqref{3.4} in the general form \eqref{2.10} for the unknowns $u=u(y)$
and $v=v(y)$ ($y \in \RR$), i.e.,
\be
{d \over dy} \begin{pmatrix} u \\ v \end{pmatrix} 
= K(u,v)
\label{3.5}
\ee
with 
\be 
K(u,v) = \begin{pmatrix} {v \over c_2(u)} \\ 
 - \alpha \, {b(u) \over c_1(u) c_2(u)} \, v + {g(u, \lam) - g(u_-, \lam) \over 
 c_1(u)}  
\end{pmatrix}, 
\quad 
g(u, \lam) := f(u) - \lam \, u, 
\label{3.6}
\ee 
while the boundary conditions take the form  
\be
\lim_{y \to \pm\infty} u(y) = u_\pm, 
\quad 
\lim_{y \to \pm\infty} v(y) = 0. 
\label{3.7}
\ee

The function $K$ in \eqref{3.6} vanishes at the {\sl equilibrium points\/} 
$(u,v) \in \RR^2$ satisfying  
\be 
g(u, \lam) = g(u_-, \lam), \quad v=0.  
\label{3.8}
\ee
In view of the assumption \eqref{3.2}, given a left-hand state $u_-$ and a speed 
$\lam$ 
there exist at most three equilibria $u$ satisfying \eqref{3.8} (including $u_-$ itself). 
Considering a trajectory leaving from $u_-$ at $-\infty$, 
we will determine whether this trajectory diverges to 
infinity or else which equilibria (if there is more than one 
equilibria) it actually connects to at $+\infty$. 
Before stating our main result (cf.~Theorem~\ref{Theorem3.3}, below) 
let us derive some fundamental inequalities satisfied by states $u_-$ 
and $u_+$ connected by a traveling wave. 


Consider the {\sl entropy dissipation}  
\be
E(u_-, u_+) := - \ab(u_-, u_+) \, \bigl(U(u_+) - U(u_-)\bigr) + 
                F(u_+) - F(u_-) 
\label{3.9}
\ee
or, equivalently, using \eqref{3.3} and \eqref{3.7} 
\be 
\aligned 
E(u_-, u_+) & = 
\int_{-\infty}^{+\infty} 
U'(u(y)) \, \bigl( - \lam \, u_y(y) + f(u(y))_y \bigr) \, dy
\\
& = 
- \int_{-\infty}^{+\infty} U''(u(y)) \, 
\bigl(-\lam \, (u(y) - u_-) + f(u) - f(u_-)\bigr) \, u_y(y) \, dy
\\
& = - \int_{u_-}^{u_+} 
           \bigl( g(z, \ab(u_-, u_+)) - g(u_-, \ab(u_-, u_+)) \bigr) \, 
           {c_2(z) \over c_1(z)} \,  dz.  
\endaligned 
\label{3.10}
\ee
In view of  
$$
\aligned 
E(u_-, u_+) 
& =\int_{-\infty}^{+\infty} U'(u) \, \bigl(
\alpha \, \bigl(b(u) \, u_y \bigr)_y 
      + \bigl(c_1(u) \, (c_2(u) \, u_y)_y\bigr)_y \bigr) \, dy 
      \\  
& = - \int_{-\infty}^{+\infty} \alpha \, U''(u) \, b(u) \, u_y^2 \, dy, 
\endaligned 
$$
we have immediately the following.

\begin{lemma}[Entropy inequality] 
\label{Lemma3.1}
If there exists a traveling wave of $\eqref{3.4}$ connecting $u_-$ to $u_+$, 
then the corresponding entropy dissipation is non-positive, 
$$
E(u_-, u_+)  \leq  E(u_-, u_-) = 0. 
$$ 
\end{lemma} 
 
From the graph of the function $f$ we define 
the functions $\vn$ and $\lna$ by  
$$ 
\lna(u) := f'\big(\vn(u)\big) = {f(u) - f\big(\vn(u)\big) \over u -\vn(u)},  
\quad u \neq 0. 
$$
We have $u \, \vn(u)<0$ and by continuity $\vn(0)=0$
and, thanks to \eqref{3.2}, the map $\vn: \RR \to \RR$ is decreasing 
and onto. 
It is invertible and its inverse function is denoted by $\vmn$. 
Observe in passing that, 
$u_-$ being kept fixed, $\lna(u_-)$ is a {\sl lower bound\/} for all shock speeds $\lam$ 
satisfying the Rankine-Hugoniot relation 
$$
- \lam \, (u_+ - u_-) + f(u_+) - f(u_-) = 0 
$$
for some $u_+$.

The properties of the entropy dissipation \eqref{3.9} are determined from the {\sl zero-entropy dissipation} 
function $\vfz$ was introduced.

\begin{lemma}[Entropy dissipation function]
\label{Lemma3.2}
There exists a decreasing function $\vfz :\RR \to \RR$ 
such that for all $u_- >0$ (for instance)  
$$ 
E(u_-, u_+) = 0 \text{ and } u_+ \neq u_- 
\quad \text{ if and only if } \quad u_+ = \vfz(u_-),  
$$
$$ 
E(u_-, u_+) < 0 \quad \text{ if and only if } \quad \vfz(u_-) < u_+ < u_-,    
$$ 
and 
$$
\vmn(u_-) < \vfz(u_-) < \vn(u_-). 
$$ 
\end{lemma}

In passing, define also the function $\vsz=\vsz(u_-)$ and the speed $\lz=\lz(u_-)$ 
by    
\be
\lz (u_-) = {f(u_-) - f(\vfz(u_-)) \over u_- - \vfz(u_-)}
           = {f(u_-) - f(\vsz(u_-)) \over u_- - \vsz(u_-)}, 
	   \quad u_- \neq 0. 
\label{3.11}
\ee
Combining Lemmas~\ref{Lemma3.1} and \ref{Lemma3.2} together we conclude that, 
if there exists a traveling wave connecting $u_-$ to $u_+$,  
necessarily 
\be  
u_+ 
\text{ belongs to the interval } [\vfz(u_-), u_-]. 
\label{3.12}
\ee
In particular, the states $u_+ >u_-$ and $u_+ < \vmn(u_-)$ 
{\sl cannot\/} be reached by a traveling wave and, therefore, 
it is not restrictive to focus on the case that three equilibria exist.  


Next, for each $u_->0$ we define the {\sl shock set} generated by 
the diffusive-dispersive model \eqref{3.1} by  
$$
\SS_\alpha (u_-) := \Big\{ u_+ \, / \, \text{ there exists a
traveling wave of \eqref{3.4} connecting $u_-$ to $u_+$} \Big\}. 
$$


\begin{theorem}[Kinetic function and shock set for general flux]
\label{Theorem3.3} 
Given a concave-convex flux-function $f$ (see \eqref{3.2}),
consider the diffusive-dispersive model \eqref{3.1} in which 
the ratio 
${\alpha = \eps / \sqrt{\delta} >0}$ is fixed.
Then, there exists a locally Lipschitz continuous and 
decreasing {\sl kinetic function}  
${\vfalpha: \RR \to \RR}$ satisfying 
\be
\aligned 
& \vn(u) \leq \vfalpha(u) < \vfz(u), \quad u < 0, 
\\
& \vfz(u) < \vfalpha(u) \leq \vn(u), \quad u > 0,  
\endaligned 
\label{3.13}
\ee
and such that 
\be
\SS_\alpha(u_-) = 
\begin{cases} 
\bigr[ u_-, \vsalpha(u_-) \bigl)  
\cup \bigl\{\vfalpha(u_-)\bigr\}, \quad u_- < 0, 
\\ 
\bigl\{\vfalpha(u_-)\bigr\} 
\cup \bigl(\vsalpha(u_-), u_-\bigr], \quad u_- > 0. 
\end{cases}
\label{3.14}
\ee
Here, the function $\vsalpha$ is defined from the kinetic function $\vfalpha$ by 
$$
{f(u) - f\big(\vsalpha(u)\big) \over u -\vsalpha(u)} 
=
{f(u) - f\big(\vfalpha(u)\big) \over u -\vfalpha(u)},   \quad u \neq 0, 
$$
with the constraint  
\be
\aligned 
& \vsz(u) < \vsalpha(u) \leq \vn(u),  \quad u < 0, 
\\
& \vn(u) \leq \vsalpha(u) < \vsz(u),  \quad u > 0.  
\endaligned 
\label{3.15}
\ee

Moreover, there exists a function 
$$
\an: \RR \to [0, +\infty), 
$$ 
called the {\sl threshold diffusion-dispersion ratio},  
which is smooth away from $u = 0$, Lipschitz continuous at $u=0$,
increasing in $u > 0$, and decreasing in $u < 0$  
with  
\be
\an(u) \sim C \, |u| \quad \text{ as } u \to 0,  
\label{3.16}
\ee
(where $C>0$ depends upon $f$, $b$, $c_1$, and $c_2$ only)
and such that 
\be
\vfalpha(u) = \vn(u) \quad \text{ when } \, \alpha \geq \an(u).   
\label{3.17}
\ee
Additionally we have 
\be
\vfalpha(u) \to \vfz(u) \quad \text{ as } \, \alpha \to 0 \quad 
\text{ for each } u \in \RR. 
\label{3.18}
\ee
\end{theorem}

%

The proof of Theorem~\ref{Theorem3.3} will be the subject of Sections~4 and 5 
below. The kinetic function $\vfalpha:\RR \to\RR$  
completely characterizes the dynamics of the nonclassical shock waves 
associated with \eqref{3.1}. In view of Theorem~\ref{Theorem3.3} one can solve the Riemann problem.  
The kinetic function $\vfalpha$ is 
decreasing and its range is limited by the functions $\vn$ and $\vfz$.  
Therefore we can solve the Riemann problem, uniquely in the class of 
nonclassical entropy solutions selected by the kinetic function 
$\vfalpha$.

The statements \eqref{3.17} and \eqref{3.18} provide us with
important qualitative properties of the nonclassical shocks: 
\begin{enumerate} 
\item The shocks leaving from $u_-$ are always classical 
if the ratio $\alpha$ is chosen to be sufficiently large 
or if $u_-$ is sufficiently small. 
\item The shocks leaving from $u_-$ are always nonclassical 
if the ratio $\alpha$ is chosen to be 
sufficiently small. 
\end{enumerate} 
Furthermore, under a mild assumption on the growth of $f$ at infinity,
one could also establish that
the shock leaving from $u_-$ are always nonclassical 
if the state $u_-$ is sufficiently large. 
(See the bibliographical notes.)


In this rest of this section we introduce some important notation 
and investigate the limiting case when the diffusion is identically zero
($\alpha = 0$). 
We always suppose that $u_- >0$ (for definiteness) and we set 
$$
u_0 = u_-. 
$$ 
The shock speed $\lam$ is regarded as a parameter allowing us to 
describe the set of attainable right-hand states. 
Precisely, given a speed in the interval 
$$
\lam \in \bigl(\lna(u_0), f'(u_0)\bigr), 
$$
there exist exactly three distinct solutions denoted by 
$u_0$, $u_1$, and $u_2$ of the equation \eqref{3.8} with  
\be
u_2 < \vn(u_0) < u_1 < u_0. 
\label{3.19}
\ee 
Recall that no trajectory exists when $\lam$ is chosen 
outside the interval limited by $\lna(u_0)$ and $f'(u_0)$.

From Lemmas~\ref{Lemma3.1} and \ref{Lemma3.2} (see \eqref{3.12}) it follows that a trajectory 
either is {\sl classical\/} if $u_0$ is connected to  
\be
\text{$u_1 \in [\vn(u_0), u_0]$ with $\lam \in \bigl[\lna(u_0), f'(u_0)\bigr]$}
\label{3.20}
\ee
or else is {\sl nonclassical\/} if $u_0$ is connected to  
\be 
\text{$u_2 \in [\vfz(u_0), \vn(u_0))$ 
with $\lam \in \bigl(\lna(u_0), \lz(u_0)\bigr]$.}
\label{3.21}
\ee


For the sake of completeness we cover here both cases of positive 
and negative dispersions. For the statements in 
Lemma~\ref{Lemma3.4} and Theorem~\ref{Theorem3.5} below {\sl only\/} we will set   
$\alpha := \eps / \sqrt{|\delta|}$ and $\eta = \sgn(\delta)= \pm 1$.   
If $(u,v)$ is an equilibrium point, 
the eigenvalues of the Jacobian matrix of the function $K(u,v)$ 
in \eqref{3.6} are found to be 
$$
\mu 
= {1 \over 2} \, 
\Bigg(- \eta \, \alpha \, { b(u) \over c_1(u) c_2(u) } \pm \sqrt{ \alpha^2 
\, {b(u)^2 \over c_1(u)^2 c_2(u)^2} 
+ 4 \eta  \, {f'(u) - \lam \over c_1(u) \, c_2(u)}} \, \Bigg). 
$$
So, we set 
\be
\aligned 
&  
\bmu(u; \lam, \alpha)
= 
{\eta \, \alpha \over 2} \, { b(u) \over c_1(u) c_2(u) } \, 
\Bigg(-1 - \eta \, \sqrt{1 + {4 \eta \over \alpha^2} \, {c_1(u) c_2(u)\over 
b(u)^2} \, (f'(u) - \lam)  } \, \Bigg), \\ 
& 
\mub(u; \lam, \alpha)
= 
{\eta \, \alpha \over 2} \, { b(u) \over c_1(u) c_2(u) } \, 
\Bigg(-1 + \eta \, \sqrt{1 + {4\eta \over \alpha^2} \, {c_1(u) c_2(u) \over 
b(u)^2} \,  (f'(u) - \lam)  } \, \Bigg). 
\endaligned 
\label{3.22}
\ee

\begin{lemma}[Nature of equilibrium points]
\label{Lemma3.4} 
Fix some values $u_-$ and $\lam$ and denote by $(u_*,0)$ any one of the 
three equilibrium points satisfying $\eqref{3.8}$. 
\begin{enumerate} 
\item If $\eta = +1$ and $f'(u_*) - \lam<0$, then $(u_*,0)$ is a {\rm stable point.} 
\item If $\eta \, (f'(u_*) - \lam)>0$, then $(u_*,0)$ is a {\rm saddle point.}  
\item If $\eta = -1$ and $f'(u_*) - \lam>0$, then $(u_*,0)$ is an {\rm unstable point.}   
\end{enumerate} 
Furthermore, in the two cases that $\eta \, ( f'(u_*) - \lam) <0$ we 
have the additional result: When  
$\alpha^2 \, b(u_*)^2 + 4\, \eta \, c_1(u_*) \, c_2(u_*) \, (f'(u_*) - 
\lam) \geq 0$   
the equilibrium is a {\rm node}, and is a {\rm spiral} otherwise. 
\end{lemma}  


For negative dispersion coefficient $\delta$, that is, when $\eta =-1$, 
we see that both $u_1$ and $u_2$ are {\sl unstable points\/} 
which no trajectory can attain at $+\infty$, while $u_1$ is a stable 
point. So, in this case, we obtain immediately:

\begin{theorem}[Traveling waves for negative dispersion]
\label{Theorem3.5} 
Consider the diffusive-dispersive model \eqref{3.1} where 
the flux satisfies \eqref{3.2}. If $\eps>0$ and $\delta<0$,  
then only classical trajectories exist.  
\end{theorem}

Some additional analysis (along similar lines) would 
be necessary to establish the existence of these classical 
trajectories and conclude that 
$$
\SS_\alpha (u_-) = \SS(u_-) := 
\begin{cases} 
\bigl[\vn(u_-), u_-\bigr],   &            u_- \geq 0  
\\
\\
\bigl[u_-, \vn(u_-)\bigr],   &            u_-  \leq 0  
\end{cases} 
\text{ when } \delta < 0, 
$$ 
which is the shock set already found in Section~2 when $\delta =0$.


We return to the case of a positive dispersion which is of main interest 
here. (From now on $\eta=+1$.) 
Since $g_u'(u, \lam)$ is positive at both $u=u_2$ and $u=u_0$, 
we have 
$$
\bmu(u_0) < 0 < \mub(u_0), \quad \bmu(u_2) < 0 < \mub(u_2), 
$$ 
and both points $u_2$ and $u_0$ are {\sl saddle.\/} 
On the other hand, since $g_u'(u_1, \lam) < 0$, the equilibrium $u_1$ 
is a stable point which may be a {\sl node\/} or a {\sl spiral.\/} 
These properties are the same as the ones already established 
for the equation with cubic flux.  
The following result is easily checked from the expressions \eqref{3.22}.

\begin{lemma}[Monotonicity properties of eigenvalues]
\label{Lemma3.6} 
In the range of parameters where $\bmu(u, \lam, \alpha)$ and 
$\mub(u; \lam,\alpha)$ remain real-valued, we have 
$$
{\del \bmu \over \del \lam} (u; \lam, \alpha) > 0,
\quad  
{\del \bmu \over \del \alpha} (u; \lam, \alpha) < 0
\quad 
{\del \mub \over \del \lam} (u; \lam, \alpha) < 0, 
$$
and, under the assumption $f'(u) - \lam >0$,  
$$
{\del \mub \over \del \alpha} (u; \lam, \alpha) < 0.
$$
\end{lemma}


To the state $u_0$ and the speed $\lam \in \bigl(\lna(u_0), \lz(u_0)\bigr)$ we associate 
the following function of the variable $u$, which 
will play an important role throughout,  
$$
G(u; u_0, \lam) := \int_{u_0}^u 
\bigl( g(z, \lam) - g(u_0, \lam) \bigr) \, {c_2(z) \over c_1(z)} \, dz. 
$$ 
Observe, using \eqref{3.10}, that the functions $G$ and $E$ are closely 
related:  
\be
G(u; u_0, \lam) = - E(u_0, u) \quad \text{ when } \quad \lam =\ab(u_0, u). 
\label{3.23}
\ee
Note also that the derivative $\del_u G(u; u_0, \lam)$ vanishes 
exactly at the equilibria $u_0$, $u_1$, and $u_2$ satisfying \eqref{3.8}.  
Using the function $G$ we rewrite now the main equations \eqref{3.5}-\eqref{3.6} 
in the form  
\be 
c_2(u) \, u_y = v, 
\label{3.24a}
\ee
\be 
c_2(u) \, v_y = - \alpha \, {b(u) \over c_1(u)} \, v + G_u'(u; u_0, \lam),
\label{3.24b}
\ee
which we will often use in the rest of the discussion. 

We collect now some fundamental properties of the function $G$.  

\begin{theorem}[Monotonicity properties of the function $G$]
\label{Theorem3.7} 
Fix some $u_0>0$ and $\lam \in \bigl(\lna(u_0), f'(u_0)\bigr)$
and consider the associated states $u_1$ and $u_2$. 
Then, the function $u \mapsto \tilde G(u):=G(u; u_0, \lam)$ satisfies the 
monotonicity properties 
$$
\aligned 
& \tilde G'(u) < 0,   \quad   u < u_2 \text{ or } u \in (u_1, u_0), 
\\
&  \tilde G'(u) > 0,   \quad  u \in (u_2,u_1) \text{ or } u > u_0.  
\endaligned 
$$ 
Moreover, if $\lam \in \bigl(\lna(u_0), \lz(u_0)\bigr)$ we have 
\be 
\tilde G(u_0) = 0 < \tilde G(u_2) < \tilde G(u_1), 
\label{3.25i}
\ee
while, if $\lam = \lz(u_0)$, 
\be
\tilde G(u_0) = \tilde G(u_2) = 0 < \tilde G(u_1)  
\label{3.25ii}
\ee
and finally, if $\lam \in \bigl(\lz(u_0), f'(u_0)\bigr)$, 
\be
\tilde G(u_2) < 0 = \tilde G(u_0) < \tilde G(u_1).
\label{3.25iii}
\ee
\end{theorem}

%


\begin{proof} The sign of $\tilde G'$ is the same as  
the sign of the function 
$$
g(u, \lam) - g(u_0, \lam)
= (u- u_0) \, \Big({f(u) - f(u_0) \over u - u_0} - \lam\Big).
$$ 
So, the sign of $\tilde G'$ is easy determined geometrically 
from the graph of the function $f$. 
To derive \eqref{3.25i}--\eqref{3.25iii} note that $\tilde G(u_0)= 0$ and (by the monotonicity
properties above) $\tilde G(u_1) > 
\tilde G(u_0)$. 
To complete the argument we only need the sign of $\tilde G(u_2)$. 
But by \eqref{3.23} we have $\tilde G(u_2) = -E(u_0, u_2)$ whose sign is given by Lemma~\ref{Lemma3.2}. 
\end{proof} 


We conclude this section with the special case that the diffusion is 
zero. Note that the shock set below {\sl is not\/} 
the obvious limit from \eqref{3.14}.

\begin{theorem}[Dispersive traveling waves]
\label{Theorem3.8} 
Consider the traveling wave equation $\eqref{3.4}$ 
in the limiting case $\alpha =0$ ({\rm not included in Theorem~\ref{Theorem3.3}})
under the assumption that the flux $f$ satisfies \eqref{3.2}. 
Then, the corresponding {\sl shock set} reduces to   
$$ 
\SS_0(u_-) = \bigl\{\vfz(u_-), u_-\bigr\}, \quad u_- \in \RR. 
$$ 
\end{theorem}

\begin{proof} Suppose that there exists a trajectory 
connecting a state $u_->0$ to a state $u_+ \neq u_-$ for the speed 
$\lam = \ab(u_-, u_+)$ and satisfying (see \eqref{3.24a}-\eqref{3.24b}) 
\be
\aligned 
&  c_2(u) \, u_y = v, \\ 
&  c_1(u) \, v_y = g(u, \lam) - g(u_-, \lam). 
\endaligned 
\label{3.26}
\ee    
Multiplying the second equation in 
\eqref{3.26} by $v/c_1(u) = c_2(u) \, u_y/ c_1(u)$, we find 
$$
{1 \over 2} \, \bigl(v^2\bigr)_y
= 
\bigl( g(u, \lam) - g(u_-, \lam) \bigr) \, {c_2(u) \over c_1(u)} \, u_y
$$
and, after integration over some interval $(-\infty, y]$,  
\be
{1 \over 2} \, v^2(y) = G(u(y); u_-, \lam), \quad y \in \RR. 
\label{3.27}
\ee 

Letting $y \to +\infty$ in \eqref{3.27} and using that $v(y) \to 0$ 
we obtain 
$$ 
G(u_+; u_-, \lam) = 0 
$$ 
which, by \eqref{3.23}, is equivalent to 
$$
E(u_+, u_-) = 0. 
$$
Using Lemma~\ref{Lemma3.2} we conclude that the right-hand state $u_+$ is 
uniquely determined, by the zero-entropy dissipation function: 
\be
u_+ = \vfz(u_-), \quad \lam = \lz(u_-). 
\label{3.28}
\ee 

Then, by assuming \eqref{3.28} and $u_->0$, Theorem~\ref{Theorem3.7} implies that the function 
$u \mapsto G(u; u_-, \lam)$ remains 
strictly positive for all $u$ (strictly) between $u_+$ and $u_-$. 
Since $v<0$ we get from \eqref{3.27} 
\be 
v(y) = - \sqrt{2 \, G(u(y); u_-, \lam)}. 
\label{3.29}
\ee
In other words, we obtain the trajectory
in the $(u,v)$ plane: 
$$ 
v = \bar v(u) = - \sqrt{2 \, G(u; u_-, \lam)}, 
\quad
u \in [u_+, u_-], 
$$
supplemented with the boundary conditions 
$$
\bar v(u_-) = \bar v(u_+) = 0. 
$$ 
Clearly, the function $\bar v$ is well-defined 
and satisfies $\bar v(u) < 0$ for all ${u \in (u_+, u_-)}$. 
Finally, based on the change of variable $y \in [-\infty, +\infty] \mapsto u=u(y)
\in [u_+, u_-]$ given by 
$$ 
dy = {c_2(u) \over \bar v(u)} \, du, 
$$
we immediately recover from the curve $v=\bar v(u)$ the (unique) trajectory 
$$
y \mapsto \bigl(u(y), v(y)\bigr). 
$$
This completes the proof of Theorem~\ref{Theorem3.8}.
\end{proof}


\section{Traveling waves corresponding to a given speed}

We prove in this section that, given $u_0$, $u_2$, and ${\lam = \ab(u_0, u_2)}$ 
in the range (see \eqref{3.21})   
\be 
u_2 \in \bigl[\vfz(u_0), \vn(u_0)\bigr), 
\quad 
\lam \in \bigl(\lna(u_0), \lz(u_0)\bigr],   
\label{4.1}
\ee
a nonclassical connection always exists if the ratio $\alpha$ 
{\sl is chosen appropriately.\/} As we will show in the next section 
this result is the key step in the proof of Theorem~\ref{Theorem3.3}.  
The main existence result proven in the present section is stated as 
follows.

\begin{theorem}[Nonclassical trajectories for a fixed speed]
\label{Theorem4.1} 
Consider two states $u_0>0$ and $u_2<0$ associated with a speed 
$$
\lam= \ab(u_0, u_2) \in \bigl(\lna(u_0), \lz(u_0)\bigr].  
$$ 
Then, there exists a unique value $\alpha \geq 0$ such that 
$u_0$ is connected to $u_2$ by a diffusive-dispersive traveling wave solution. 
\end{theorem}


By Lemma~\ref{Lemma3.4}, $u_0$ is a saddle point and 
we have $\mub(u_0) > 0$ and from Theorem~\ref{Theorem2.2}
 it 
follows that there are two trajectories leaving from $u_0$ at $y = 
-\infty$, both of them satisfying 
\be
\lim_{y\to -\infty} {v(y)\over {u(y)-u_0}}=\mub(u_0; \lam, \alpha)\, c_2(u_0). 
\label{4.2}
\ee
One trajectory approaches $(u_0,0)$ in the quadrant 
$Q_1= \bigl\{u>u_0, \, v>0\bigr\}$, the other in the quadrant 
$Q_2 = \bigl\{u<u_0, \, v<0\bigr\}$. On the other hand, 
$u_2$ is also a saddle point and there exist two
trajectories reaching $u_2$ at $y=+\infty$, both of them satisfying
\be
\lim_{ y \to + \infty} { v(y) \over u(y) - u_2 } 
 = \bmu(u_2; \lam, \alpha) \, c_2(u_2). 
\label{4.3}
\ee
One trajectory approaches $(u_2,0)$ in the quadrant 
$Q_3 = \bigl\{u>u_2, \, v<0\bigr\}$, the other in the quadrant 
$Q_4 = \bigl\{u<u_2, \, v>0 \bigr\}$. 
 

\begin{lemma}
\label{Lemma4.2}
A traveling wave solution connecting $u_0$ to $u_2$ must 
leave the equilibrium $(u_0,0)$ at $y= -\infty$ in the quadrant $Q_2$,
and reach $(u_2,0)$ in the quadrant $Q_3$ at $y= +\infty$. 
\end{lemma}

\begin{proof} Consider the trajectory leaving from the quadrant 
$Q_1$, that is, satisfying $u>u_0$ and $v>0$ in a neighborhood of 
the point $(u_0,0)$. By contradiction, suppose it would reach 
the state $u_2$ at $+\infty$. Since $u_2 < u_0$ 
by continuity there would exist $y_0$ such that 
$$
u(y_0) = u_0.
$$
Multiplying \eqref{3.24b} by $u_y = v /c_2(u)$ we find 
$$
\bigl( v^2/ 2 \bigr)_y 
+ 
\alpha \, {b(u)\over c_1(u) \, c_2(u)}  \, v^2
= 
G_u'(u; u_0, \lam) \, u_y.  
$$
Integrating over $(-\infty, y_0]$ we arrive at 
\be 
{v^2(y_0) \over 2}  
+ \alpha \int_{-\infty}^{y_0} v^2 \, {b(u)\over c_1(u) \, c_2(u)}  \, dy 
= G(u(y_0); u_0, \lam) = 0.
\label{4.4}
\ee
Therefore $v(y_0)=0$ and, since $u(y_0) = u_0$, 
a standard uniqueness theorem for the Cauchy problem associated
with \eqref{3.24a}-\eqref{3.24b} implies that $u \equiv u_0$ and $v \equiv
0$ on $\RR$. This contradicts the assumption that the trajectory
would connect to $u_2$ at $+\infty$.


The argument around the equilibrium $(u_2,0)$ is somewhat different. 
Suppose that the trajectory satisfies $u<u_2$ and $v>0$ in a neighborhood of 
the point $(u_2,0)$. There would exist some value $y_1$ achieving a 
{\sl local minimum,\/} that is, such that
$$
u(y_1)<u_2, \quad u_y(y_1)=0, \quad u_{yy}(y_1) \geq 0.
$$
From \eqref{3.24a} we would obtain  
$v(y_1)=0$ and, by differentiation of \eqref{3.24a}, 
$$
v_y(y_1)=u_{yy}(y_1) \,c_2(u(y_1)) \geq 0.
$$ 
Combining the last two relations with \eqref{3.24b} we would obtain 
$$
G_u'(u(y_1); u_0, \lam) \geq 0
$$ 
which is in contradiction with Theorem~\ref{Theorem3.7} since $u(y_1) < u_2$
and ${G_u'(u(y_1); u_0, \lam) < 0}$. 
\end{proof} 

Next, we determine some intervals in which the traveling waves 
are always monotone.

\begin{lemma}
\label{Lemma4.3}
Consider a trajectory $u = u(y)$ leaving from $u_0$ at 
$-\infty$ and denote by $\underline \xi$ the largest value such that  
$u_1 < u(y) \leq  u_0$ for all $y \in (-\infty, \underline \xi)$ 
and $u(\underline \xi) = u_1$. Then, we have 
$$
u_y < 0 \quad \text{ on the interval } (-\infty, \underline \xi). 
$$
Similarly, if $u = u(y)$ is a trajectory connecting to $u_2$ at 
$+\infty$, denote by $\overline \xi$ the smallest value such that  
$u_2  \leq u(y) <  u_1$ for all $y \in (\overline \xi, +\infty)$ 
and $u(\overline \xi) = u_1$. 
Then, we have 
$$
u_y < 0 \quad \text{ on the interval } (\overline \xi, +\infty).
$$  
\end{lemma}

In other words, a trajectory cannot change its monotonicity before 
reaching the value $u_1$.

\begin{proof} 
We only check the first statement, the proof of the second 
one being similar. By contradiction, 
there would exist $y_1\in (-\infty, \underline \xi)$ such that 
$$
u_y(y_1)=0, \quad u_{yy}(y_1)\geq 0, 
\quad 
u_1 < u(y_1) \leq u_0. 
$$
Then, using the equation \eqref{3.24b} 
would yield $G_u'(u(y_1); u_0, \lam) \geq 0$, which is in 
contradiction with the monotonicity properties in Theorem~\ref{Theorem3.7}. 
\end{proof}

\begin{proof}[Proof of Theorem~\ref{Theorem4.1}] 
For each $\alpha\geq 0$ we consider the orbit leaving from $u_0$ 
and satisfying $u<u_0$ and $v<0$ in a neighborhood of $(u_0,0)$. 
This trajectory reaches the line  $\bigl\{ u = u_1 \bigr\}$ 
for the ``first time'' 
at some point denoted by $(u_1, V_-(\alpha))$.
In view of Lemma~\ref{Lemma4.3} this part of trajectory is {\sl the graph of a function\/}  
$$
[u_1, u_0] \owns u \mapsto v_-(u; \lam,\alpha)  
$$
with of course $v_-(u_1; \lam, \alpha) = V_-(\alpha)$. 
Moreover, by standard theorems on differential equations, $v_-$ is  
a smooth function with respect to its argument 
$(u; \lam,\alpha) \in [u_1, u_0]\times\bigl(\lna(u_0),
\lz(u_0)\bigr]\times [0, +\infty)$. 

Similarly, for each $\alpha\geq 0$ we consider 
the orbit arriving at $u_2$ and satisfying $u>u_2$ and $v<0$ 
in a neighborhood of $(u_2,0)$. 
This trajectory reaches the line $\bigl\{ u = u_1 \bigr\}$ for the ``first time'' 
{\sl as $y$ decreases from $+ \infty$\/} at some point $(u_1, V_+(\alpha))$. 
By Lemma~\ref{Lemma4.3} this trajectory is the {\sl graph of a function\/}  
$$ 
[u_2, u_1] \owns u \mapsto v_+(u; \lam,\alpha). 
$$ 
The mapping $v_+$ depends smoothly upon $(u,\lam,\alpha) \in 
[u_2, u_1]\times\bigl(\lna(u_0), \lz(u_0)\bigr]\times [0, +\infty)$.

For each of these curves $u\mapsto  v_-(u)$ and $u\mapsto  v_+(u)$ 
we derive easily from \eqref{3.24a}-\eqref{3.24b}
a differential equation in the $(u,v)$ plane: 
\be
v(u) \, {dv \over du}(u) + \alpha \, {b(u) \over c_1(u)} \, v(u) = G_u'(u, u_0, \lam). 
\label{4.5}
\ee
Clearly, the function
$$
\aligned 
\alpha \in [0, +\infty)  \mapsto 
W(\alpha): & = v_+(u_1; \lam,\alpha) - v_-(u_1; \lam,\alpha) 
\\
& = V_+(\alpha) - V_-(\alpha)
\endaligned 
$$
measures the distance (in the phase plane) between the two trajectories at
$u=u_1$. 
Therefore, the condition $W(\alpha)=0$ characterizes the
traveling wave solution of interest connecting $u_0$ to $u_2$. 
The existence of a root for the function $W$ is obtained as follows. 

\vskip.3cm 

\noindent{\sl Case 1: } Take first $\alpha=0$.

Integrating \eqref{4.5} with $v=v_-$ over the interval $[u_1, u_0]$
yields 
$$ 
{1 \over 2} (V_-(0))^2 
= G(u_1; u_0,\lam) - G(u_0; u_0,\lam)  = G(u_1; u_0,\lam), 
$$
while integrating \eqref{4.5} with $v=v_+$ over the interval $[u_2, u_1]$ 
gives 
$$
{1 \over 2} (V_+(0))^2 = G(u_1; u_0,\lam) - G(u_2; u_0,\lam). 
$$
When $\lam \neq \lz(u_0)$, 
since $G(u_2; u_0,\lam)>0$ (Theorem~\ref{Theorem3.7}) and
$V_\pm(\alpha)<0$ (Lemma~\ref{Lemma4.3}) we conclude that $W(0)>0$.
When $\lam = \lz(u_0)$ we have $G(u_2; u_0,\lam)=0$ and ${W(0)=0}$. 
\vskip.3cm 


{\sl Case 2:} Consider next the limit $\alpha \to + \infty$.  

On one hand, since $v_-<0$, for $\alpha>0$ we get in the same way as
in Case~1 
\be
{1 \over 2} (V_-(\alpha))^2 < G(u_1;u_0,\lam). 
\label{4.6}
\ee 
On the other hand, dividing \eqref{4.5} by $v=v_+$ and 
integrating over the interval $[u_2, u_1]$ we find 
$$
V_+(\alpha) = - \alpha\int_{u_2}^{u_1} {b(u) \over c_1(u)} \, du 
      + \int_{u_2}^{u_1} {G_u'(u; u_0, \lam)\over v_+(u)} \, du.
$$
Since $v= c_2(u) \, u_y \leq 0$ and $G_u'(u) \geq 0$ in the interval $[u_2,
u_1]$ we obtain
\be
V_+(\alpha) \leq - \kappa \,\alpha \, (u_1 - u_2),
\label{4.7}
\ee
where $\kappa = \inf_{u\in[u_2, u_1]} b(u)/c_1(u)>0$.
Combining \eqref{4.6} and \eqref{4.7} and choosing $\alpha$ to be sufficiently 
large, we conclude that
$$
W(\alpha) = V_+(\alpha) - V_-(\alpha) < 0. 
$$  
\vskip.1cm 

Hence, by the intermediate value theorem there exists at least one
value $\alpha$ such that 
$$
W(\alpha) =0, 
$$
which establishes the existence of a trajectory connecting $u_0$ to $u_2$. 
Thanks to Lemma~\ref{Lemma4.3} it satisfies $u_y < 0$ {\sl globally.\/}


The uniqueness of the solution is established as follows. 
Suppose that there would exist two orbits $v=v(u)$ and $v^* = v^*(u)$
associated with distinct values $\alpha$ and $\alpha^* > \alpha$,
respectively. Then, Lemma~\ref{Lemma3.6} would imply that 
$$
\mub(u_0; \lam, \alpha^*) < \mub(u_0; \lam, \alpha), 
\quad 
\bmu(u_2; \lam, \alpha^*) < \bmu(u_2; \lam, \alpha).
$$  
So, there would exist $u_3 \in (u_2, u_0)$ satisfying 
$$
v(u_3) = v^*(u_3), \quad {dv^* \over du} (u_3) \geq {dv \over du}(u_3).
$$   
Comparing the equations \eqref{4.5} satisfied by both $v$ and $v^*$, we get 
\be 
v(u_3) \, \Big({dv \over du}(u_3)-{dv^* \over du}(u_3)\Big)
= 
(\alpha^* - \alpha) \, {b(u_3) \over c_1(u_3)} \, v(u_3).
\label{4.8}
\ee
Now, since $v(u_3)\neq 0$ (the connection with the third critical 
point $(u_1,0)$ is impossible) we obtain a contradiction, as 
the two sides of \eqref{4.8} have opposite signs. 
This completes the proof of Theorem~\ref{Theorem4.1}. 
\end{proof}


\begin{remark}
\label{Remark4.4} 
It is not difficult to see also that, in the proof of Theorem~\ref{Theorem4.1},   
\be
\text{$\alpha \mapsto V_-(\alpha)$ is non-decreasing}
\label{4.9i}
\ee
and 
\be 
\text{$\alpha \mapsto V_+(\alpha)$ is decreasing.}
\label{4.9ii}
\ee
In particular, the function $W(\alpha): = V_+(\alpha) - V_-(\alpha)$ 
is decreasing. 
\end{remark} 


\begin{theorem}[Threshold function associated with nonclassical shocks]
\label{Theorem4.5} 
Consider the function $A= A(u_0, u_2)$ which is
the unique value $\alpha$ for which there is a nonclassical 
traveling wave connecting $u_0$ to $u_2$ (Theorem~\ref{Theorem4.1}).
It is defined for $u_0>0$ and $u_2 < 0$ with 
$u_2 \in \bigl[\vfz(u_0), \vn(u_0)\bigr)$ 
or, equivalently, $u_0 \in \bigl[\vfz(u_2), \vmn(u_2) \bigr)$. 
Then we have the following two properties: 
\begin{enumerate} 
\item The function $A\bigl(u_0, u_2\bigr)$ is increasing 
in $u_2$ and maps 
$[\vfz(u_0),\vn(u_0)\bigl)$ onto some interval of the form 
$\bigl[0, \an(u_0)\bigr)$ where $\an(u_0) \in (0, + \infty]$. 
\item 
The function $A$ is also increasing in $u_0$
and maps the interval $[\vfz(u_2), \vmn(u_2))$ onto the interval 
$\bigl[0, \an(\vmn(u_2)) \bigr)$. 
\end{enumerate} 
\end{theorem}

Later (in Section~6) the function $A$ will also determine the range 
in which {\sl classical shocks\/} exist. 
From now on, we refer to the function $A$ as the {\sl critical  
diffusion-dispersion ratio.} On the other hand, 
the value $\an(u_0)$ is called the 
{\sl threshold diffusion-dispersion ratio} at $u_0$. 
Nonclassical trajectories leaving from $u_0$ exist if and only if  
$\alpha < \an(u_0)$.

Observe that, in Theorem~\ref{Theorem4.5}, we have $A(u_0, u_2) \to 0$ 
when $u_2 \to \vfz(u_0)$, which is exactly the desired property \eqref{3.18} 
in Theorem~\ref{Theorem3.3}. 


\begin{proof} We will only prove the first statement, the proof of 
the second one being completely similar.  
Fix $u_0>0$ and $u_2^*<u_2<u_0$ so that 
$$
\lna(u_0)<\lam={f(u_2)-f(u_0) \over u_2 -u_0}<
\lam^*={f(u_2^*)-f(u_0) \over u_2^* -u_0}\leq\lz(u_0).
$$
Proceeding by contradiction we assume that 
$$
\alpha^* := A( u_0, u_2^*) \geq \alpha := A(u_0, u_2).
$$
Then, Lemma~\ref{Lemma3.6} implies  
$$
\mub(u_0; \lam, \alpha) \geq \mub(u_0; \lam, \alpha^*) 
> \mub(u_0; \lam^*, \alpha^*). 
$$ 
Let $v= v(u)$ and $v^*= v^*(u)$ be the solutions of \eqref{4.5} associated with 
$\alpha$ and $\alpha^*$, respectively, 
and connecting $u_0$ to $u_2$, and $u_0$ to $u_2^*$, respectively.
Since $u_2^* < u_2$, by continuity there must exist some state 
$u_3 \in (u_2, u_0)$ such that 
$$
v(u_3) = v^*(u_3), \quad {dv^* \over du}(u_3)\geq {dv \over du}(u_3). 
$$ 
On the other hand, in view of \eqref{4.5} which is satisfied by both $v$ and $v^*$
we obtain 
$$ 
v(u_3) \, \Big({dv^* \over du}(u_3)-{dv \over du}(u_3)\Big)+
v(u_3)(\alpha^*-\alpha) \, {b(u_3) \over c_1(u_3)} 
=
(\lam^* - \lam) \, (u_0 - u_3) \, {c_2(u_3) \over c_1(u_3)} , 
$$
which leads to a contradiction since the left-hand side is non-positive and 
the right-hand side is positive.  
This completes the proof of Theorem~\ref{Theorem4.5}.  
\end{proof}


We complete this section with some important asymptotic properties 
(which will establish \eqref{3.16}-\eqref{3.17} in Theorem~\ref{Theorem3.3}).

\begin{theorem}
\label{Theorem4.6} 
The threshold diffusion-dispersion ratio satisfies the following
two properties: 
\begin{enumerate}
\item $\an(u_0)< + \infty$ for all $u_0$.
\item There exists a traveling wave connecting $u_0$ to
$u_2=\vn(u_0)$ for the value $\alpha =\an(u_0)$.
\end{enumerate}
\end{theorem}

\begin{proof} Fix $u_0>0$. 
According to Theorem~\ref{Theorem4.1}, given $\lam\in(\lna(u_0), \lz(u_0)]$  
there exists a nonclassical trajectory, denoted by $u \mapsto v(u)$,  
connecting $u_0$ to some $u_2$ with 
\be 
\lam ={{f(u_2)-f(u_0)}\over {u_2-u_0}}, \quad u_2<\vn(u_0),
\quad \alpha=A( u_0, u_2).
\label{4.10}
\ee 
On the other hand, choosing any state $u_0^*>u_0$ and setting 
$$
\lam^*={{f(u_0^*)-f(u_1^*)}\over {u_0^*-u_1^*}}, 
\quad 
u_1^*=\vn(u_0),
$$
it is easy to check from \eqref{3.22} that, for all $\alpha^*$ 
{\sl sufficiently large\/}, $\bmu(u_1^*; \lam^*, \alpha^*)$
remains real with  
$$
\bmu(u_1^*; \lam^*, \alpha^*) < 0.
$$ 
Then, consider the trajectory $u \mapsto v^*(u)$ arriving at $u_1^*$ and satisfying
$$
\lim_{u \to u_1 \atop u>u_1} {v^*(u) \over u - u_1^* } 
= \bmu(u_1^*; \lam^*, \alpha^*) \, c_2(u_1^*) 
< 0. 
$$ 
Two different situations should be distinguished.

\vskip.3cm 

\noindent{\sl  Case 1 : } 
The curve $v^*=v^*(u)$ crosses the curve $v=v(u)$ at some point $u_3$ where
$$ 
u_1^*<u_3<u_0, \quad v(u_3)
=
v^*(u_3), 
\quad {dv\over du}(u_3)\geq {dv^*\over du}(u_3).
$$ 
Using the equation \eqref{4.5} satisfied by both $v$ and $v^*$ we get
$$ 
\aligned 
v(u_3) \, \Big({dv^*\over du}(u_3)-{dv\over du}(u_3)\Big) +
(\alpha^*-\alpha) \, {b(u_3) \over c_1(u_3)} \, 
 v(u_3) 
& = G_u'(u_3; u_0^*, \lam^*) - G_u'(u_3; u_0, \lam)
\\
& < 0. 
\endaligned 
$$
In view of our assumptions, since 
$v(u_3)<0$ we conclude that $\alpha<\alpha^*$ in this first case.

\vskip.3cm 

\noindent{\sl  Case 2 : } $v^*=v^*(u)$ does not cross the curve $v=v(u)$ 
on the interval $(u_1^*, u_0)$. 

Then, the trajectory $v^*$ crosses the $u$-axis at some point
$u_4 \in (u_1^*, u_0]$. 
Integrating the equation \eqref{4.5} for the function $v$   
on the interval $[u_2, u_0]$ we obtain
$$
\alpha \int_{u_0}^{u_2} \, {b(u) \over c_1(u)} \, v(u) \, du =
G(u_2; u_0, \lam) - G(u_0; u_0, \lam).
$$
On the other hand, integrating \eqref{4.5} for the solution $v^*$ over $[u_1^*, u_4]$ 
we get
$$
\alpha^* \int_{u_4}^{u_1^*} {b(u) \over c_1(u)} \, v^*(u) \, du = 
G(u_1^*; u_0^*, \lam^*) - G(u_4; u_0^*, \lam^*).
$$
Since, by our assumption in this second case,  
$$
\int_{u_0}^{u_2} {b(u) \over c_1(u)} \, v(u) du > \int_{u_4}^{u_1^*}
{b(u) \over c_1(u)}  \, v^*(u)du, 
$$
we deduce from the former two equations that 
$$
\alpha \leq \alpha^* {G(u_2; u_0, \lam) - G(u_0; u_0, \lam)
  \over {G(u_1^*; u_0^*, \lam^*) - G(u_4; u_0^*, \lam^*)}}\leq C \alpha^*,
$$
where $C$ is a constant {\sl independent\/} of $u_2$. 
More precisely, $u_2$ describes a small neighborhood of $\vn(u_0)$, 
while $u_0^*$, $u_1^*$, $u_4$, and $\lam^*$ remain fixed.

Finally, we conclude that in both cases 
$$
A(u_0, u_2)\leq C' \, \alpha^*, 
$$ 
where $\alpha^*$ is sufficiently large (the condition depends on $u_0$ only) 
and $C'$ is independent of the right-hand state $u_2$ under 
consideration. Hence, we have obtained an {\sl upper bound\/} for the function 
$u_2 \mapsto A(u_0, u_2)$. This completes the proof of the first 
statement in the theorem.

The second statement is a consequence of the fact that $A(u_0, u_2)$
remains bounded as $u_2$ tends to $\vn(u_0)$ and of the continuity of the 
traveling wave $v$ with respect to the parameters $\lam$ and $\alpha$, 
i.e., with obvious notation 
$$
v(.; \lna(u_0), \an(u_0)) = \lim_{ u_2 \to \vn(u_0) } v(.; \lam(u_0, u_2), A(u_0, u_2)).
$$
\end{proof}


The function $\an= \an(u_0)$ 
maps the interval $(0, + \infty)$ onto some interval 
$[\ban, \anb]$ where $0 \leq \ban \leq \anb \leq + \infty$. 
The values $\ban$ and $\anb$ correspond to {\sl lower\/} and {\sl upper
bounds\/} for the threshold ratio, respectively. 
The following theorem shows that the range of the function
$\an(u_0)$, in fact, has the form $\bigl[0, \anb\bigr]$.

\begin{theorem}
\label{Theorem4.7} 
With the notation in Theorem~\ref{Theorem4.5} 
the asymptotic behavior of $\an(u_0)$ as $u_0 \to 0$ is given by 
\be 
\an(u_0) \sim \kappa \, u_0, \quad  
\kappa := {c_1(0) c_2(0) \over {4 \, b(0)}} \, \sqrt{3 f'''(0)} >0. 
\label{4.11}
\ee  
\end{theorem}

Note that of course \eqref{3.2} implies that $f'''(0) >0$. 
In particular, Theorem~\ref{Theorem4.7} shows that $\ban(0) = \an(0) = 0$. 
Theorem~\ref{Theorem4.7} is the only instance where the assumption 
$f'''(0) \neq 0$ (see \eqref{3.2}) is needed. In fact, if
this assumption is dropped one still have $\an(u_0) \to 0$ as $u_0 \to 0$. 
(See the bibliographical notes.) 
 

\begin{proof} To estimate $\an$ near the origin
we compare it with the corresponding critical function
$\ane$ determined explicitly from the third-order Taylor 
expansion $f^*$ of ${f=f(u)}$ at ${u = 0}$. (See \eqref{4.16} below.) 
We rely on the results in Section~3, especially the formula \eqref{2.26} which 
provides the threshold ratio explicitly for the cubic flux.

Fix some value $u_0>0$ and the speed ${\lam =\lna(u_0)}$ so that, 
with the notation introduced earlier, ${u_2 = u_1 = \vn(u_0)}$. 
Since ${f'''(0) \neq 0}$ it is not difficult to see that 
$$
u_2 = \vn(u_0) = - (1 + O(u_0)) \, {u_0 \over 2}  
$$ 
(as is the case for the cubic flux ${f(u) = u^3}$). 
A straightforward Taylor expansion for the function 
$$
G(u) := G(u; u_0, \lna(u_0))
$$
yields 
$$ 
\aligned 
G(u) - G(u_2) 
&= G(u) - G(\vn(u_0)) 
\\  
& = 
  {(u-u_2)^3 \over 24} \, 
                \Big(f'''(0) \,{c_2(0) \over c_1(0)} 
		\, (3 \, u_2 + u) + O(|u_2|^2 + |u|^2) \Big). 
\endaligned 
$$ 
Since, for all $u \in [u_2, u_0]$ 
$$
4 \,u_2 < u + 3 \, u_2 < u_0 + 3 \, u_2 = u_2 \, (1 + O(u_0)), 
$$ 
we arrive at 
\be
\Big|G(u) - G(u_2) - f'''(0) \,{c_2(0) \over c_1(0)}\, (u + 3\, u_2) \,{(u-u_2)^3 \over 24}\Big| 
\leq 
     C \, u_0 \, |u + 3\, u_2| \, (u - u_2)^3.
\label{4.12}
\ee
Now, given $\eps>0$, we can assume that $u_0$ is sufficiently
small so that 
\be 
\aligned   
(i) & \quad -{u_0 \over 2} \, (1+\eps) \leq u_2 \leq - {u_0 \over 2} \,
(1-\eps), 
\\
(ii) & \quad (1-\eps) \, {b(0)\over c_1(0)} \leq {b(u) \over c_1(u)} \leq (1 + \eps) 
\, {b(0)\over c_1(0)},  
       \quad u \in [u_2, u_0], 
\\       
(iii) & \quad c_j(0) \, (1-\eps) \leq c_j(u) \leq c_j(0) \, (1 + \eps),  
       \quad u \in [u_2, u_0], \quad j=1,2. 
\endaligned 
\label{4.13}
\ee

Introduce next the flux-function  
\be
f_*(u) = k\, {u^3 \over 6}, \quad k = (1 + \eps) \, f'''(0), 
\quad u \in \RR. 
\label{4.14}
\ee
Define the following (constant) functions 
$$
b^*(u)=b(0), \quad c_1^*(u) = c_1(0), \quad c_2^*(u) = c_2(0). 
$$ 
To these functions we can associate a function $G_*$ by 
the general definition in Section~4. We are interested in 
traveling waves associated with the functions 
$f_*$, $b^*$, $c_1^*$, and $c_2^*$, and connecting the 
left-hand state $u_0^*$ given by 
$$
u_0^* = -2 \, u_2
$$
to the right-hand state $u_2$ (which will also correspond
to the traveling wave associated with $f$). 

The corresponding function 
$$
G_*(u) := G_*(u; u_0^*, \lna(u_0^*))
$$
satisfies 
\be
G_*(u) - G_*(u_2) = f'''(0) \,{c_2(0) \over c_1(0)} \, 
{(1+\epsilon) \over 24} \ (u + 3 \, u_2) \,
(u-u_2)^3.
\label{4.15}
\ee
In view of Remark~2.6 the threshold function $\ane$ associated with
$f_*$, $b^*$, $c_1^*$, and $c_2^*$ is  
\be 
\ane(u_0^*) = {\sqrt{3 k} \, c_1(0) c_2(0) \over 4 \, b(0)} \, u_0^*. 
\label{4.16}
\ee
By Theorem~\ref{Theorem4.6}, 
for the value $\alpha^* := \ane(u_0^*)$ there exists also a
traveling wave trajectory   
connecting $u_0^*$ to $u_2^* := u_2$, which we denote by $v^*= v^*(u)$.
By definition, in the phase plane it satisfies
\be
v^* \, {dv^* \over du}(u) + \alpha^* \, {b^*(u) \over c_1^*(u)} \, v^*(u) 
= G_*'(u),
\label{4.17}
\ee
with 
$$
G_*'(u) = \bigl( f_*(u) - f_*(u_0^*) - f_*'(u_2) \, (u-u_0^*) \bigr) \, 
{c_2^*(u) \over c_1^*(u)}. 
$$ 


We consider also the traveling wave trajectory $u \mapsto v=v(u)$ 
connecting $u_0$ to $u_2$ which is associated with the 
data $f$, $b$, $c_1$, and $c_2$ and the threshold value $\alpha:=\an(u_0)$. 
We will now establish lower and upper bounds on $\an(u_0)$; see 
\eqref{4.23} and \eqref{4.24} below. 

\vskip.3cm 

\noindent{\sl Case 1 : } First of all, in the easy case that 
$\an(u_0) \, (1-\eps) \leq \ane(u_0^*)$, 
we immediately obtain by \eqref{4.16} and then \eqref{4.13} 
$$
\aligned 
\an(u_0) 
& \leq (1 + 2 \, \eps) \, \ane(u_0^*)
=  
(1 + 2 \, \eps) \, \sqrt{3 k} \, {c_1(0) c_2(0) \over 4 \, b(0)} 
\, u_0^*  
\\
& \leq (1 + 2 \, \eps) \, \sqrt{3 k} \, { c_1(0) c_2(0) \over 4 \, b(0)}
\, u_0 \, (1 + \eps) 
\\ 
& \leq (1 + C \, \eps) \,  \sqrt{3 f'''(0)} \, 
{c_1(0) c_2(0) \over 4 \, b(0)} \, u_0,
\endaligned 
$$
which is the desired upper bound for the threshold function. 


\vskip.3cm 

\noindent{\sl Case 2 : } Now, assume that $\an(u_0) \, (1-\eps) > \ane(u_0^*)$ 
and let us derive a similar inequality on $\an(u_0)$. 
Since $G'(u_2) = G_*'(u_2) = 0$, $G''(u_2) = G_*''(u_2) = 0$, 
and
$$
v(u_2) = v(u_2^*) = 0, \quad 
{dv \over du}(u_2) <0, \quad
{dv^* \over du}(u_2)<0, 
$$
it follows from the equation 
$$
{dv \over du}(u) + \alpha \, {b(u) \over c_1(u)} = {G_u'(u; u_0, 
\lam) \over v(u)} 
$$
by letting $u \to u_2$ that 
$$
{dv \over du}(u_2)
= 
- \ane(u_0) \, {b(u_2) \over c_1(u_2)}   
< 
    {-1 \over 1 -\eps} \, \ane(u_0^*) \, {b(0) \over c_1(0)} \, (1 - \eps)  
= {dv^* \over du}(u_2). 
$$ 
This tells us that in a neighborhood of the point $u_2$ the curve 
$v$ is locally below the curve $v^*$.

Suppose that the two trajectories meet for the ``first time'' 
at some point $u_3 \in (u_2, u_0]$, 
so 
$$
v(u_3)=v^*(u_3) \quad \text{ with } 
{dv\over du}(u_3) \geq {dv^* \over du}(u_3).
$$ 
{}From the equations \eqref{4.5} satisfied by $v = v(u)$ and $v^* = v^*(u)$, we
deduce 
$$
{1 \over 2}  \, v(u_3)^2 + \alpha \, \int_{u_2}^{u_3} v(u) \, {b(u) \over c_1(u)} \, du 
= G(u_3) - G(u_2), 
$$
and 
$$
{1 \over 2}  \, v^*(u_3)^2  + \alpha^* \, \int_{u_2}^{u_3} v^*(u) \, {b(u) \over c_1(u)} \, du 
= G_*(u_3) - G_*(u_2), 
$$
respectively. Subtracting these two equations 
and using \eqref{4.12} and \eqref{4.15}, we obtain 
\be 
\aligned 
\alpha \, \int_{u_2}^{u_3} v(u) \, {b(u) \over c_1(u)} \, du 
- \alpha^* \, \int_{u_2}^{u_3} & v^*(u) \, {b^*(u) \over c^*_1(u)} \, du  
\\
& = G(u_3)- G(u_2) - \bigl(G_*(u_3) - G_*(u_2)\bigr) 
\\
&\geq (O(u_0) - C \, \eps) \, (u_3 + 3\, u_2) \, (u_3 - u_2)^3.
\endaligned 
\label{4.18}
\ee
But, by assumption the curve 
$v$ is locally below the curve $v^*$ so that the left-hand side of 
\eqref{4.18} is negative, 
while its right-hand side of \eqref{4.18} is positive if one chooses $u_0$ 
sufficiently small. 
We conclude that the two trajectories intersect only at $u_2$, 
which implies that $u_0^*\leq u_0$ and thus  
\be
\int_{u_2}^{u_0} |v(u)| \, du > \int_{u_2}^{u_0^*} |v^*(u)| \, du.
\label{4.19}
\ee

On the other hand we have by \eqref{4.13} 
\be
\aligned 
\an(u_0) \, {b(0) \over c_1(0)} \, (1- \eps) \, \int_{u_2}^{u_0} |v(u)| \, du
& \leq 
\an(u_0) \, 
\int_{u_2}^{u_0} {b(u) \over c_1(u)} \,  |v(u)| \, du 
\\
& = G(u_2) - G(u_0).
\endaligned 
\label{4.20}
\ee
Now, in view of the property (i) in \eqref{4.13} we have 
$$
|3 \, u_2 + u_0| \leq {u_0\over 2} \, (1+ 3 \, \eps)
\leq |u_2| \, {1 + 3\eps \over 1-\eps},
\quad 
|u_2 - u_0| \leq {u_0\over 2} \,(3 + \eps) \leq 
|u_2| \, {3 + \eps \over 1-\eps}.
$$
Based on these inequalities we deduce from \eqref{4.12} that
\be 
G(u_2) - G(u_0) 
\leq f'''(0) \,{c_2(0) \over c_1(0)} \, {9 \, |u_2|^4 \over 8} \, 
(1+ C \, \eps). 
\label{4.21}
\ee  
Concerning the second curve, $v^* = v^*(u)$, we have 
\be
\aligned 
\ane(u_0^*) \, {b(0) \over c_1(0)}  \, \int_{u_2}^{u_0^*} |v^*(u)| \, du 
& = G_*(u_2) - G_*(u_0^*)\\ 
& = f'''(0) \,{c_2(0) \over c_1(0)} \, {9 \, |u_2|^4 \over 8} \, ( 1 + \eps)  
\endaligned 
\label{4.22}
\ee
by using \eqref{4.15}.

Finally, combining \eqref{4.19}--\eqref{4.22} we conclude that for every $\eps$ 
and for all sufficiently small $u_0$: 
\be
\aligned 
\an(u_0) 
& \leq (1 + C \, \eps) \,   \ane(u_0^*) 
\\
& \leq (1 + C \, \eps) \, \sqrt{3 f'''(0)} \, {c_1(0) c_2(0) 
\over {4 \, b(0)}} \, u_0,
\endaligned 
\label{4.23}
\ee
which is the desired upper bound.  
Exactly the same analysis as before but based on the cubic function
$f_*(u) = k \, u^3$ with $k = (1 - \eps) \, f'''(0)$ 
(exchanging the role played by $f_*$ and $f$, however) 
we can also derive the following inequality  
\be
\an(u_0) \geq \sqrt{3 f'''(0)} \, {c_1(0) c_2(0) 
\over {4 \, b(0)}} \, u_0 \, (1 - C \, \eps).
\label{4.24}
\ee
The proof of Theorem~\ref{Theorem4.7} is thus completed since $\eps$ is arbitrary in
\eqref{4.23} and \eqref{4.24}. 
\end{proof}  


\section{Traveling waves corresponding to a given dif\-fusion-dispersion ratio}

Fixing the parameter $\alpha$, 
we can now complete the proof of Theorem~\ref{Theorem3.3} by identifying 
the set of right-hand state attainable from $u_0$ by 
classical trajectories. We rely here mainly 
on Theorem~\ref{Theorem4.1} (existence of the nonclassical trajectories) 
and Theorem~\ref{Theorem4.5} (critical function).

Given $u_0>0$ and $\alpha>0$, a classical traveling wave must  
connect $u_-=u_0$ to $u_+=u_1$ for some shock speed 
$\lam\in(\lna(u_0),f'(u_0))$. According to Theorem~\ref{Theorem4.5}, 
to each pair of states $(u_0, u_2)$ we can associate 
the critical ratio $A(u_0, u_2)$. Equivalently, to 
each left-hand state $u_0$ and each speed $\lam$,   
we can associate a critical value ${B(\lam, u_0)= A(u_0, u_2)}$. 
The mapping 
$$
\lam \mapsto B(\lam, u_0)  
$$
is defined and decreasing from 
the interval $\bigl[\lna(u_0), \lz(u_0)\bigr]$ onto 
$\bigl[0,\an(u_0)\bigr]$. It admits an inverse   
$$
\alpha \mapsto \Lamalpha(u_0), 
$$
defined from the interval $\bigl[0, \an(u_0)\bigr]$ 
onto $\bigl[\lna(u_0), \lz(u_0)\bigr]$. 
By construction, given any $\alpha \in \bigl(0, \an(u_0)\bigr)$
there exists a nonclassical traveling trajectory (associated with the shock
speed  
$\Lamalpha(u_0)$) leaving from $u_0$ and solving the equation with the
prescribed value $\alpha$.

It is natural to {\sl extend the definition of the function\/} $\Lamalpha(u_0)$ 
to arbitrary values $\alpha$ by setting 
$$
\Lamalpha(u_0) = \lna(u_0), \quad \alpha \geq \an(u_0). 
$$
The nonclassical traveling waves are considered here 
when $\alpha$ is a fixed parameter. 
So, we define the {\sl kinetic function\/} for nonclassical shocks, 
$$
(u_0, \alpha) \mapsto \vfalpha(u_0) = u_2,
$$
where $u_2$ denotes the right-hand state of the
nonclassical trajectory, so that  
\be
{f(u_0) - f(u_2) \over u_0 - u_2} = \Lamalpha(u_0). 
\label{5.1}
\ee 
Note that 
$\vfalpha(u_0)$ makes sense for all $u_0 >0$ but $\alpha < \an(u_0)$.

\begin{theorem}
\label{Theorem5.1} 
For all $u_0 > 0 $ and $\alpha >0$ and for every speed satisfying 
$$
\Lamalpha(u_0) < \lam \leq f'(u_0),
$$
there exists a unique traveling wave connecting $u_-=u_0$ to $u_+=u_1$.
Moreover, for ${\alpha \geq \an(u_0)}$ 
there exists a traveling wave connecting $u_- = u_0$ to $u_+ = u_1$
for all
$$
{\lam\in \bigl[ \lna(u_0), f'(u_0) \bigr]}.
$$ 
\end{theorem}

\begin{proof} We first treat the case 
$\alpha\leq \an(u_0)$ and 
$\lam \in \bigl(\Lamalpha(u_0), f'(u_0)\bigr]$. 
Consider the curve $u \mapsto v_-(u; \lam,\alpha)$
defined on $[u_1, u_0]$ that was introduced earlier in the proof of 
Theorem~\ref{Theorem4.1}. 
We have either $v_-(u_1; \lam,\alpha)=0$ and the proof is completed, or
else $v_-(u_1; \lam,\alpha)<0$. In the latter case, the function $v_-$
is a solution of \eqref{4.5} that extends further on the left-hand side of
$u_-$ in the phase plane. On
the other hand, this curve cannot cross the nonclassical trajectory 
$u \mapsto v(u)$ connecting $u_-=u_0$ to $u_+=\vfalpha(u_0)$. 
Indeed, by Lemma~\ref{Lemma3.6} we have 
$$
\mub(u_0; \lam, \alpha) < \mub(u_0; \Lamalpha(u_0), \alpha). 
$$
If the two curves would cross, there would exist $u^* \in
(\vfalpha(u_0), u_1)$ 
such that 
$$
v(u^*) = v_-(u^*) \quad \text{ and } \quad 
{dv\over du}(u^*)\leq {dv_-\over du}(u^*).
$$ 
By comparing the equations \eqref{4.5} satisfied by these two
trajectories we get 
\be
v(u^*) \, \Big({dv\over du}(u^*)- {dv_-\over du}(u^*)\Big)
=
\bigl(\lam-\Lamalpha(u_0)\bigr) \, (u^*-u_0) \, {c_2(u^*) \over 
c_1(u^*)}.
\label{5.2}
\ee
This leads to a contradiction since the right-hand side of \eqref{5.2} is positive 
while the left-hand side is negative. We conclude that the function $v_-$ must 
cross the $u$-axis at some point $u_3$ with $u_2 < \vfalpha(u_0) < u_3 <u_1$. 
The curve $u \mapsto v_-(u, \lam,\alpha)$ on the interval $[u_3, u_0]$
corresponds to a solution $y \mapsto u(y)$ 
in some interval $(-\infty, y_3]$ with $u_y(y_3)=0$ and 
\be
u_{yy} (y_3) = {g\bigl(u(y_3),\lam\bigr) - g\bigl(u_0,\lam\bigr)
\over c_1(u(y_3)) \, c_2(u(y_3))} 
= {G_u'(u_3; u_0, \lam) \over c_2(u_3)^2},
\label{5.3}
\ee
which is positive by Theorem~\ref{Theorem3.7}. Thus $u_{yy}(y_3) > 0$ and necessarily 
$u(y) > u_3$ for $y > y_3$. Indeed, assume that there exists $y_4 > y_3$,
such that $u(y_4) = u(y_3) = u_3$. Then, multiplying \eqref{3.24b} by $v_-/c_2$ and
integrating over $[y_3,y_4]$, we obtain  
$$
{1 \over 2 } \, v_-^2(y_4) + \alpha \, \int_{y_3}^{y_4} {b(u) \over c_1(u) \, 
c_2(u)} \,  v_-^2 \, dy 
= G(u_3; u_0, \lam) - G(u_3; u_0, \lam) = 0. 
$$
This would means that $u(y)=u_3$ for all $y$, which is excluded since
$u_-=u_1$.

Now, since $u \leq u_0$ we see that $u$ is bounded. Finally, by 
integration over the interval $(-\infty, y]$ we obtain 
$$
{1 \over 2 } \, v_-^2(y) + \alpha \, \int_{-\infty}^y {b(u) \over c_1(u) \, 
c_2(u)} \, v_-^2 \, dy = G(u(y)) - G(u_0), 
$$
which implies that $v$ is bounded and that the function
$u$ is defined on the whole real line $\RR$. When 
$y \to +\infty$ the trajectory $(u,v)$ converges to a critical point 
which can only be $(u_1,0)$.


Consider now the case $\alpha> \an(u_0)$. 
The proof is essentially same as the one given above. However, we
replace the 
nonclassical trajectory with the curve $u \mapsto v_+(u)$ defined on
the interval $[u_2, u_1]$. For each $\lam$ fixed in
$(\lna(u_0), f'(u_0))$ (since $\alpha> \an(u_0)$) 
and thanks to Remark~4.4, the function, 
$W = V_+ - V_-$ (defined in the proof of Theorem~\ref{Theorem4.1},
with $v_-(u;\lam, \alpha)$ and $v_+(u;\lam, \alpha)$ and extended to  
$\lam \in (f'(u_2), f'(u_0))$)
satisfies $W(\alpha)<0$. On the left-hand side of $u_1$,
with the same argument as in the first part above, we can prove 
that the extension of $v_-$ does not intersect $v_+$ 
and must converge to $(u_1, 0)$. Finally, the case $\lam = \lna(u_0)$ is reached 
by continuity.  This completes the proof of Theorem~\ref{Theorem5.1}. 
\end{proof}


\begin{theorem}
\label{Theorem5.2}
If $\lna(u_0)<\lam<\Lamalpha(u_0)$ there is no traveling
wave connecting $u_-=u_0$ to $u_+=u_1$. 
\end{theorem}

\begin{proof}
Assume that there exists  a traveling wave connecting $u_0$ to $u_1$.
As in Lemma~\ref{Lemma4.2}, we prove easily that such a curve must approach $(u_0, 0)$ 
from the quadrant $Q_1$ and coincide with the function $v_-$ on the
interval $[u_1, u_0]$. 
On the other hand, as in the proof of Theorem~\ref{Theorem5.1}, we see that this curve
does not cross the nonclassical trajectories. On the other hand, we have 
$$
\mub(u_0; \lam, \alpha) \geq \mub\bigl(u_0; \Lamalpha(u_0),\alpha\bigr), 
$$ 
thus, the classical curve remains ``under" the nonclassical one. So 
we have 
$$
{v_- \bigl(\vfalpha(u_0)\bigr) < v\bigl(\vfalpha(u_0)\bigr)},
$$
where $u \mapsto (u, v(u))$ denotes the nonclassical trajectory. 
Assume now that the curve $(u, v_-(u))$ meets the $u$-axis for the first time 
at some point $(u_3, 0)$ with $u_3 < \vfalpha(u_0) < u_2$. 
The previous curve defined on $[u_3, u_0]$ corresponds to a solution $y \mapsto u(y)$
defined on some interval $(-\infty, y_3]$ with $u_y(y_3)=0$ and 
$u_{yy}(y_3) \geq 0$. 
Thus $v_y(y_3)$ satisfies \eqref{5.3} and is negative (Lemma~\ref{Lemma4.3}). 
This implies that $u_{yy}(y_3)<0$ which is a contradiction. 
Finally, the trajectory remains under the $u$-axis for $u<u_2$,
and cannot converge to any critical point.
\end{proof} 


According to Theorem~\ref{Theorem5.1}  
the kinetic function can now be extended to all values of $\alpha$ by setting 
\be
\vfalpha(u_0) = \vn(u_0),  \quad \alpha \geq \an(u_0). 
\label{5.4}
\ee
Finally we have:

\begin{theorem}
\label{Theorem5.3} 
{\rm (Monotonicity of the kinetic function.)} For each $\alpha>0$ 
the mapping $u_0 \mapsto \vfalpha(u_0)$ is decreasing.
\end{theorem}

\begin{proof}
Fix $u_0>0$, $\alpha>0$, $\lam = \Lamalpha(u_0)$ and $u_2= \vfalpha(u_0)$. First suppose that $\alpha \geq \an(u_0)$.
Then, for all $u_0^* >u_0$, since $\vn$ is known to be strictly 
monotone, it is clear that
$$
\vfalpha(u_0^*) \leq  \vn(u_0^*) < \vn(u_0)=\vfalpha(u_0).
$$
Suppose now that $\alpha < \an(u_0)$.
Then,
for
$u_0^* >u_0$ in a neighborhood of $u_0$, the speed 
$\lam^*={{f(u_0^*)-f(u_2)}\over {u_0^*-u_2}}$ 
satisfies $\lam^*\in \bigl(\lna(u_0^*), \lz(u_0^*)\bigr)$. Then, there 
exists a nonclassical traveling wave connecting $u_0^*$ to $u_2$ 
for some $\alpha^*= A(u_0^*, u_2)$. 
The second statement in Theorem~\ref{Theorem4.5} gives $\alpha^*>\alpha$. 
Since the function $\Lamalpha$ is decreasing (by the first statement in Theorem~\ref{Theorem4.5})
we have 
$\Lambda_{\alpha_*}(u_0^*)<\Lamalpha(u_0^*)$
and thus $\vfalpha(u_0^*) < u_2 = \vfalpha(u_0)$
and the proof of 
Theorem~\ref{Theorem5.3} is completed. 
\end{proof}

\begin{proof}[Proof of Theorem~\ref{Theorem3.3}] 
Section~5 provides us with the existence of nonclassical trajectories, 
while Theorems~\ref{Theorem5.1} and \ref{Theorem5.2} are concerned with classical trajectories. 
These results prove that the shock set is given by \eqref{3.14}. 
By standard theorems on solutions of ordinary differential equations
the kinetic function is smooth in the region  
$\bigl\{ \alpha \leq \an(u_0) \bigr\}$ while it coincides with the 
(smooth) function $\vn$ in the region $\bigl\{ \alpha \geq \an(u_0) \bigr\}$. 
Additionally, by construction the kinetic function is continuous 
along $\alpha = \an(u_0)$. This proves that $\vf$ is 
Lipschitz continuous on each compact interval. 
On the other hand, the monotonicity of the kinetic function
is provided by Theorem~\ref{Theorem5.3}. 
The asymptotic behavior was the subject of Theorem~\ref{Theorem4.7}. 
\end{proof}

\begin{remark}
\label{Remark5.4} 
To a large extend the techniques presented 
in this paper extend to {\sl systems of equations,\/} in 
particular to a classical model of elastodynamics and phase transitions.   
The corresponding traveling wave solutions $(v,w)= (v(y), w(y))$ must solve 
$$
\aligned 
- s & \, v_y - \Sigma\bigl( w,w_y, w_{yy} \bigr)_y = \bigl( \mu(w) \, v_y \bigr)_y, 
\\
s & \, w_y + v_y = 0,
\endaligned 
$$ 
where $s$ denotes the speed of the traveling wave, 
$\Sigma$ is the total stress function, and $\mu(w)$ is the viscosity 
coefficient. When $\Sigma$ is given,   
some integration with respect to $y$ we arrive at 
$$
\aligned 
- s & \, ( v - v_- ) - \sigma(w) + \sigma(w_-) - \mu(w) \, v_y
= {\lam'(w) \over 2} \, w_y^2 - \bigl( \lam(w) \, w_y \bigr)_y,  
\\
s & \, (w- w_-) + v - v_- = 0,
\endaligned 
$$ 
where $(v_-, w_-)$ denotes the upper left-hand limit and $\lam(w)$ 
the capillarity coefficient. Using the second equation above we can 
eliminate the unknown $v(y)$, namely 
\be
\lam(w)^{1/2} \, \bigl( \lam(w)^{1/2} \, w_y \bigr)_y
+ \mu(w) \, v_y 
= s^2  \, ( w - w_- ) - \sigma(w) + \sigma(w_-),  
\label{5.5}
\ee
which has precisely the structure of the equation \eqref{3.4} studied in the 
present paper, so that  
most of our results extend to the equation \eqref{5.5}.  
\end{remark}


\section*{Acknowledgments}  

This paper is an updated version of some material issued from the book \cite{LeFloch-book}
and was written for the UFRJ Winter School on Nonlinear Analysis, held in Rio de Janeiro in August 2009. 
I am very grateful to M.F. Elbert, W. Neves, and A. Pazoto 
for their invitation and hospitality, and for giving me the opportunity to 
give a short-course. The author was supported by a DFG-CNRS collaborative grant between France and Germany
on ``Micro-Macro Modeling and Simulation of Liquid-Vapor Flows'', as well as by 
the Centre National de la Recherche Scientifique (CNRS) and 
the Agence Nationale de la Recherche (ANR) via the grant 06-2-134423. 


\end{document}